\documentclass[a4paper, 12pt, leqno]{amsart}
\usepackage{amsmath}
\usepackage{amscd}
\usepackage{amssymb}
\usepackage{amsthm}
\usepackage{amsfonts}
\usepackage{latexsym}
\usepackage{multirow}

\newtheorem{lemma}{{\sc Lemma}}[section]
\newtheorem{corollary}[lemma]{{\sc Corollary}}
\newtheorem{proposition}[lemma]{{\sc Proposition}}
\newtheorem{theorem}[lemma]{{\sc Theorem}}
\newtheorem{remark}[lemma]{{\sc Remark}}

\numberwithin{equation}{section}

\def\Gg{{\mathfrak{g}}}
\def\Gh{{\mathfrak{h}}}
\def\Gk{{\mathfrak{k}}}

\def\BA{{\mathbb{A}}}

\def\BC{{\mathbb{C}}}
\def\BF{{\mathbb{F}}}

\def\BQ{{\mathbb{Q}}}
\def\BZ{{\mathbb{Z}}}

\def\CG{{\mathcal G}}
\def\CH{{\mathcal H}}

\def\CP{{\mathcal P}}

\def\Ad{{\mathop{\rm Ad}\nolimits}}
\def\ad{{\mathop{\rm ad}\nolimits}}

\def\Fr{{\mathop{\rm Fr}\nolimits}}

\def\Gr{{\mathop{\rm{Gr}}\nolimits}}
\def\Har{{\mathop{\rm Har}\nolimits}}
\def\Hom{\mathop{\rm Hom}\nolimits}

\def\Image{\mathop{\rm Im}\nolimits}

\def\Lie{\mathop{\rm Lie}\nolimits}
\def\lo{{\mathop{\rm long}\nolimits}}
\def\Mod{\mathop{\rm Mod}\nolimits}

\def\sh{{\mathop{\rm short}\nolimits}}
\def\Spec{{\rm{Spec}}}

\def\trace{{\rm{Tr}}}

\title[the center of a quantized enveloping algebra]
{The center of a quantized enveloping algebra at an even root of unity
}

\author{Toshiyuki TANISAKI
}
\dedicatory{To Shin-ichi Kato on his 60th birthday}
\address{
Department of Mathematics, Osaka City University, 3-3-138, Sugimoto, Sumiyoshi-ku, Osaka, 558-8585 Japan}
\email{tanisaki@sci.osaka-cu.ac.jp}

\begin{document}
\maketitle
\renewcommand{\thefootnote}{}
\footnotetext
{This work was supported by JSPS KAKENHI 24540026.}

\begin{abstract}
We will give an explicit description of the center of 
the  De Concini-Kac type specialization of a quantized enveloping algebra at an even root of unity.
The case of an odd root of unity was already dealt with by De Concini-Kac-Procesi.
Our description in the even case is similar to but a little more complicated than the odd case.

\end{abstract}

\section{Introduction}
The representation theory of the De Concini-Kac type specialization of a quantized enveloping algebra at a root of unity was initiated by De Concini-Kac \cite{DK}.
It is quite different from and much more complicated than the generic parameter case.
A special feature at a root of unity is that 
the center of the quantized enveloping algebra becomes much larger than the generic parameter case. An explicit description of the center of the De Concini-Kac type specialization at a root of unity was given by De Concini-Kac-Procesi \cite{DKP} when the order of the root of unity is odd.
In this paper we give a similar description of the center in the even order case.
We point out that there already exists partial results in the even order case in Beck \cite{Beck}.

Let $U_q=U_q(\Delta)$ be the simply-connected quantized enveloping algebra associated to a finite irreducible root system $\Delta$
(the Cartan part is isomorphic to the group algebra of the weight lattice).
For $z\in\BC^\times$ we denote by $U_z=U_z(\Delta)$ the specialization at $q=z$ of the De Concini-Procesi form of $U_q$.
Set $d=1$ (resp.\ 2, resp.\ 3) when $\Delta$ is of type $A, D, E$ (resp.\ $B, C, F$, resp.\ $G_2$).
We note that $U_z$ coincides with the specialization of the more standard De Concini-Kac form if $z^{2d}\ne1$.
Let $\ell$ be a positive integer, and let $\zeta\in\BC^\times$ be a primitive $\ell$-th root of 1.
We assume that the order of $\zeta^2$ is greater than $d$.

Assume that $\ell$ is odd.
If $\Delta$ is of type $G_2$, we also assume that $\ell$ is prime to 3.
In this case De Concini-Kac-Procesi \cite{DKP} gave an explicit description of the center $Z(U_\zeta)$ as explained in the following.
Denote by $Z_{\Har}(U_\zeta)$ the subalgebra of $Z(U_\zeta)$ consisting of reductions of central elements of $U_q$ contained in the De Concini-Procesi form.
Then we have a Harish-Chandra type isomorphism $Z_{\Har}(U_\zeta)\cong\BC[2P]^W$, where $P$ is the weight lattice, $W$ is the Weyl group, and the action of $W$ on the group algebra $\BC[2P]$ is a twisted one.
On the other hand we have a Frobenius homomorphism $F:U_1\to U_\zeta$, which is an injective Hopf algebra homomorphism whose image is contained in $Z(U_\zeta)$.
Set $Z_{\Fr}(U_\zeta)=\Image(F)$.
Then De Concini-Kac-Procesi proved that the canonical homomorphism
\[
Z_{\Fr}(U_\zeta)\otimes_
{Z_{\Fr}(U_\zeta)\cap Z_{\Har}(U_\zeta)}
Z_{\Har}(U_\zeta)
\to
Z(U_\zeta)
\]
is an isomorphism.
They have also given the following geometric description of $Z(U_\zeta)$ (see also De Concini-Procesi \cite{DP}).
Denote by $G$ the connected simply-connected simple algebraic group over $\BC$ with root system $\Delta$.
Take Borel subgroups $B^+$ and $B^-$ of $G$ such that $B^+\cap B^-$ is a maximal torus of $G$.
We set $H=H(\Delta)=B^+\cap B^-$.
Denote by $N^\pm$ the unipotent radical of $B^\pm$.
Define a subgroup $K=K(\Delta)$ of $B^+\times B^-$ by
\[
K=\{(tx,t^{-1}y)\in B^+\times B^-\mid
t\in H, x\in N^+, y\in N^-\}.
\]
Then we have
\begin{align*}
&Z_{\Fr}(U_\zeta)
\cong U_1\cong\BC[K],\qquad
Z_{\Har}(U_\zeta)
\cong
\BC[H/W],\\
&
Z_{\Fr}(U_\zeta)
\cap
Z_{\Har}(U_\zeta)
\cong
\BC[H/W],
\end{align*}
and the morphisms
$K\to H/W$, $H/W\to H/W$ corresponding to the embeddings $Z_{\Fr}(U_\zeta)
\cap
Z_{\Har}(U_\zeta)
\subset 
Z_{\Fr}(U_\zeta)$
and
$Z_{\Fr}(U_\zeta)
\cap
Z_{\Har}(U_\zeta)
\subset 
Z_{\Har}(U_\zeta)$
are given by $(g_1,g_2)\mapsto \Ad(G)((g_1g_2^{-1})_s)\cap H$, and $[t]\mapsto[t^\ell]$, respectively.
Here, $g_s$ for $g\in G$ denotes the semisimple part of $g$ in its Jordan decomposition.
In conclusion, we obtain 
\[
Z(U_\zeta)\cong\BC[K\times_{H/W}{H/W}].
\]

Now assume that $\ell$ is even, or $\Delta$ is of type $G_2$ and $\ell$ is an odd multiple of 3.
We can similarly define $Z_{\Har}(U_\zeta(\Delta))$ as a subalgebra of $Z(U_\zeta(\Delta))$ isomorphic to $\BC[2P]^W\cong\BC[H(\Delta)/W]$.
However, it is a more delicate problem to define $Z_{\Fr}(U_\zeta(\Delta))$.
We have an injective Hopf algebra homomorphism $F:U_\varepsilon(\Delta')\to U_\zeta(\Delta)$,
where $\varepsilon\in\{\pm1\}$, $\Delta'\in\{\Delta,\Delta^\vee\}$ are determined from $\Delta$ and $\ell$.
Here, $\Delta^\vee$ denotes the set of coroots.
This $F$ is a dual version of the Frobenius homomorphism for the Lusztig forms defined in \cite{Lbook}.
In the case $\Delta$ is of type $G_2$ and $\ell$ is an odd multiple of 3 we have $\varepsilon=1$, $\Delta'=\Delta^\vee$ and $\Image(F)\subset Z(U_\zeta(\Delta))$.
In the case $\ell$ is even and $\varepsilon=1$, $U_1(\Delta')$ is commutative, but $\Image(F)$ is not a subalgebra of $Z(U_\zeta(\Delta))$.
In  the case $\varepsilon=-1$ $U_{-1}(\Delta')$ is non-commutative.
We define $Z_{\Fr}(U_\zeta(\Delta))$ to be the intersection $\Image(F)\cap Z(U_\zeta(\Delta))$.
Then the conclusion is similar to the odd order case.
Namely, 
the canonical homomorphism
\[
Z_{\Fr}(U_\zeta(\Delta))\otimes_
{Z_{\Fr}(U_\zeta(\Delta))\cap Z_{\Har}(U_\zeta(\Delta))}
Z_{\Har}(U_\zeta(\Delta))
\to
Z(U_\zeta(\Delta))
\]
turns out to be an isomorphism.
Moreover, 
we have
\begin{align*}
&Z_{\Fr}(U_\zeta(\Delta))
\cong \BC[K(\Delta')/\Gamma],\qquad
Z_{\Har}(U_\zeta(\Delta))
\cong
\BC[H(\Delta)/W],\\
&
Z_{\Fr}(U_\zeta(\Delta))
\cap
Z_{\Har}(U_\zeta(\Delta))
\cong
\BC[H(\Delta')/W],
\end{align*}
where, 
$\Gamma$ is a certain finite group acting on the algebraic variety $K(\Delta')$, and
the morphism
$K(\Delta')/\Gamma\to H(\Delta')/W$ is induced by $K(\Delta')\to H(\Delta')/W$.
The definition of 
$H(\Delta)/W\to H(\Delta')/W$ is more involved and omitted here.
In conclusion, we obtain 
\[
Z(U_\zeta(\Delta))\cong\BC[(K(\Delta')/\Gamma)\times_{H(\Delta')/W}{H(\Delta)/W}].
\]

The proof is partially similar to that for the odd order case in De Concini-Kac-Procesi \cite{DKP}.
However, some arguments are simplified using certain bilinear forms arising from the Drinfeld pairing.
We also note that we have avoided the usage of  quantum coadjoint orbits in this paper.
We hope to investigate the quantum coadjoint orbits in the even order case in the near future
since they should be indispensable in developing the representation theory.

In dealing with the case $\varepsilon=-1$ we use $U_{-1}(\Delta')^\Gamma\cong U_{1}(\Delta')^\Gamma$.
We establish it using a result of \cite{KKO} relating $U_{-q}$ with $U_{q}$.
I would like to thank Masaki Kashiwara for explaining it to me.

\section{Quantized enveloping algebras}
\subsection{}
Let $\Delta$ be a (finite) reduced irreducible root system in a vector space $\Gh^*_\BQ$ over $\BQ$ (we assume that $\Gh^*_\BQ$ is spanned by the elements of $\Delta$).
We denote by $W$ the Weyl group.
We fix a $W$-invariant positive definite symmetric bilinear form 
\begin{equation}
\label{eq:bilinear}
(\;,\;):\Gh_\BQ^*\times\Gh_\BQ^*\to\BQ.
\end{equation}
For $\alpha\in\Delta$ we set $\alpha^\vee=2\alpha/(\alpha,\alpha)\in\Gh^*_\BQ$.
Then $\Delta^\vee=\{\alpha^\vee\mid\alpha\in\Delta\}$ is also an irreducible root system in a vector space $\Gh^*_\BQ$.
Set
\begin{align*}
&Q=\sum_{\alpha\in\Delta}\BZ\alpha,\qquad
Q^\vee=\sum_{\alpha\in\Delta}\BZ\alpha^\vee,\\
&P=\{\lambda\in\Gh_\BQ^*\mid(\lambda,\alpha^\vee)\in\BZ\;\;(\alpha\in\Delta)\},\\
&P^\vee=\{\lambda\in\Gh_\BQ^*\mid(\lambda,\alpha)\in\BZ\;\;(\alpha\in\Delta)\}.
\end{align*}
Take a set $\Pi=\{\alpha_i\}_{i\in I}$ of simple roots of $\Delta$, and denote by $\Delta^+$  the corresponding set of positive roots of $\Delta$.
Then $\Pi^\vee=\{\alpha_i^\vee\}_{i\in I}$ is a set of simple roots of $\Delta^\vee$, and
$\Delta^{\vee +}=\{\alpha^\vee\mid\alpha\in\Delta^+\}$ is  the corresponding set of positive roots of $\Delta^{\vee}$.
We set
\begin{align*}
&Q^+=\sum_{\alpha\in\Delta^+}\BZ_{\geqq0}\alpha,\\
&
P^+=\{\lambda\in\Gh_\BQ^*\mid(\lambda,\alpha^\vee)\in\BZ_{\geqq0}\;\;(\alpha\in\Delta^+)\}.
\end{align*}
For $i\in I$ let $s_i\in W$ be the corresponding simple reflection.
We denote the standard partial order on $W$ by $\geqq$.
We denote by $\Delta_\sh$ (resp.\ $\Delta_\lo$) the set of short (resp.\ long) roots.
In our convention we have $\Delta_\sh=\Delta_\lo=\Delta$ if $\Delta$ is of type $A, D, E$.
We set 
\begin{align*}
&
d=\frac{(\alpha,\alpha)}{(\beta,\beta)}
\quad(\alpha\in\Delta_{\lo},\; \beta\in\Delta_{\sh}),\\
&
d_\alpha=\frac{(\alpha,\alpha)}{(\beta,\beta)}
\quad(\alpha\in\Delta,\; \beta\in\Delta_{\sh}), 
\qquad
d_i=d_{\alpha_i}\quad(i\in I).
\end{align*}
Define $\rho\in P\cap \frac12Q$ by
$(\rho,\alpha_i^\vee)=1\;(i\in I)$.
Define $\tilde{\rho}\in\frac12Q^\vee$ by
$
\tilde{\rho}=\frac12\sum_{\alpha\in\Delta^+}d_\alpha\alpha^\vee.
$
We have $\rho=\frac{(\alpha,\alpha)}2\tilde{\rho}$ for $\alpha\in\Delta_{\sh}$.

For $n\in\BZ_{\geqq0}$ we set
\[
[n]_t=\frac{t^n-t^{-n}}{t-t^{-1}}\in\BZ[t, t^{-1}],\qquad
[n]_t!=[n]_t[n]_{t-1}\cdots[1]_t\in\BZ[t,t^{-1}].
\]

\subsection{}
Let $\BF=\BQ(q)$ be the rational function field in the variable $q$, and set
\[
q_\alpha=q^{d_\alpha}
\quad(\alpha\in\Delta), 
\qquad
q_i=q_{\alpha_i}\quad(i\in I).
\]
We denote by $U=U(\Delta)$ the corresponding simply-connected quantized enveloping algebra over $\BF$, i.e., $U$ is an associative algebra over $\BF$ generated by the elements
$k_\lambda\;(\lambda\in P)$, $e_i, f_i\;(i\in I)$ satisfying the fundamental relations
\begin{align*}
&k_0=1, \qquad k_\lambda k_\mu=k_{\lambda+\mu}\quad(\lambda, \mu\in P),
\\
&k_\lambda e_ik_\lambda^{-1}=q_i^{(\lambda,\alpha_i^\vee)}e_i
\qquad(\lambda\in P,\;i\in I),
\\
&k_\lambda f_ik_\lambda^{-1}=q_i^{-(\lambda,\alpha_i^\vee)}f_i
\qquad(\lambda\in P,\;i\in I),
\\
&e_if_j-f_je_i=\delta_{ij}
(k_i-k_i^{-1})/(q_i-q_i^{-1})
\qquad(i, j\in I),
\\
&\sum_{n=0}^{1-a_{ij}}(-1)^ne_i^{(1-a_{ij}-n)}e_je_i^{(n)}=0
\qquad(i,j\in I,\,i\ne j),
\\
&\sum_{n=0}^{1-a_{ij}}(-1)^nf_i^{(1-a_{ij}-n)}f_jf_i^{(n)}=0
\qquad(i,j\in I,\,i\ne j),
\end{align*}
where $k_i=k_{\alpha_i}\;(i\in I)$,\quad $a_{ij}=(\alpha^\vee_i,\alpha_j)\;\;(i, j\in I)$, \quad$e_i^{(n)}=e_i^n/[n]_{q_i}!,\;f_i^{(n)}=f_i^n/[n]_{q_i}!\;\;(i\in I, n\in\BZ_{\geqq0})$.
Note that the above definition of $U(\Delta)$ does not depend on the choice of the symmetric bilinear form $(\,,\,)$.

We regard $U$ as a Hopf algebra by
\begin{align*}
&\Delta(k_\lambda)=k_\lambda\otimes k_\lambda\quad(\lambda\in P),\\
&\Delta(e_i)=e_i\otimes1+k_i\otimes e_i,\qquad
\Delta(f_i)=f_i\otimes k_i^{-1}+1\otimes f_i\quad(i\in I),\\
&\varepsilon(k_\lambda)=1\quad(\lambda\in P),\qquad
\varepsilon(e_i)=
\varepsilon(f_i)=0\quad(i\in I),\\
&S(k_\lambda)=k_\lambda^{-1}\quad(\lambda\in P),\qquad
S(e_i)=-k_i^{-1}e_i,\quad
S(f_i)=-f_ik_i\quad(i\in I).
\end{align*}

Define subalgebras $U^0, U^+, U^-, U^{\geqq0}, U^{\leqq0}$ of $U$ by
\begin{align*}
&U^0=
\langle k_\lambda \mid
\lambda\in P\rangle,\qquad
U^+=
\langle e_i \mid
i\in I\rangle,\qquad
U^-=
\langle f_i \mid
i\in I\rangle,\\
&
U^{\geqq0}=
\langle k_\lambda,\; e_i \mid
\lambda\in P,\;i\in I\rangle,\qquad
U^{\leqq0}=
\langle k_\lambda,\; f_i \mid
\lambda\in P,\;i\in I\rangle.
\end{align*}
We have $U^0=\bigoplus_{\lambda\in P}\BF k_\lambda$, and the multiplication of $U$ induces isomorphisms
\begin{align*}
&
U^+\otimes U^0\otimes U^-\cong U^-\otimes U^0\otimes U^+\cong U,\\
&
U^+\otimes U^0\cong U^0\otimes U^+\cong U^{\geqq0},\qquad
U^-\otimes U^0\cong U^0\otimes U^-\cong U^{\leqq0}
\end{align*}
of vector spaces.

We denote by $U_{\ad}$ the $\BF$-subalgebra of $U$ generated by 
$k_\lambda\;(\lambda\in Q)$, $e_i, f_i\;(i\in I)$.
We also set 
\begin{align*}
&U_{\ad}^0=
\langle k_\lambda \mid
\lambda\in Q\rangle,\qquad
\\
&
U_{\ad}^{\geqq0}=
\langle k_\lambda,\; e_i \mid
\lambda\in Q,\;i\in I\rangle,\qquad
U_{\ad}^{\leqq0}=
\langle k_\lambda,\; f_i \mid
\lambda\in Q,\;i\in I\rangle.
\end{align*}
Then we have
\begin{align*}
&
U^+\otimes U_{\ad}^0\otimes U^-\cong U^-\otimes U_{\ad}^0\otimes U^+\cong U_{\ad},\\
&
U^+\otimes U_{\ad}^0\cong U_{\ad}^0\otimes U^+\cong U_{\ad}^{\geqq0},\qquad
U^-\otimes U_{\ad}^0\cong U_{\ad}^0\otimes U^-\cong U_{\ad}^{\leqq0}.
\end{align*}

We denote by $\Mod({U}_\ad)$ the category of 
finite-dimensional ${U}_\ad$-modules $M$ with weight space decomposition
$
M=\bigoplus_{\lambda\in P}M_\lambda$, where
\[
M_\lambda=\{m\in M\mid k_im=q_i^{(\lambda,\alpha_i^\vee)}m
\quad(i\in I)\}.
\]

\subsection{}
The modified quantized enveloping algebra $\dot{U}=\dot{U}(\Delta)$ is defined as follows (see Lusztig \cite{Lbook}).
For $\gamma\in Q$ set 
$U_{\ad,\gamma}=\{u\in U_\ad\mid k_iuk_i^{-1}=q_i^{(\gamma,\alpha_i^\vee)}u\;(i\in I)\}$.
For $\lambda, \mu\in P$ we set 
\[
{}_\lambda \overline{U}_\mu=
U_\ad/
(
\sum_{i\in I}(k_i-q_i^{(\lambda,\alpha_i^\vee)})U_\ad
+
\sum_{i\in I}U_\ad(k_i-q_i^{(\mu,\alpha_i^\vee)})
)
\]
(note ${}_\lambda \overline{U}_\mu=0$ unless $\lambda-\mu\in Q$), and let
${}_\lambda p_\mu:U_\ad\to{}_\lambda \overline{U}_\mu$ be the natural map.
For $\lambda\in P$ set $1_\lambda={}_\lambda p_\lambda(1)$.
Set 
\[
\dot{U}=\bigoplus_{\lambda,\mu\in P}{}_\lambda \overline{U}_\mu.
\]
Then $\dot{U}$ is an associative algebra (without 1) by 
\[
{}_\lambda p_\mu(x){}_{\lambda'} p_{\mu'}(y)=
\begin{cases}
{}_\lambda p_{\mu'}(xy)
\quad&(\mu=\lambda')\\
0&(\mu\ne\lambda')
\end{cases}
\]
for $x\in U_{\ad,\lambda-\mu},\;y\in U_{\ad,\lambda'-\mu'}$.
Moreover, $\dot{U}$ is a $U_\ad$-bimodule by
\[
u\cdot
{}_\lambda p_\mu(x)
\cdot u'
={}_{\lambda+\gamma} p_{\mu-\gamma'}(uxu')
\qquad
(x\in U_{\ad,\lambda-\mu}, u\in U_{\ad,\gamma}, u'\in U_{\ad,\gamma'}).
\]
Then we have an isomorphism
\[
\bigoplus_{\lambda\in P} (U^-\otimes U^+)\cong\dot{U}
\qquad
((u_\lambda\otimes u'_\lambda)_{\lambda\in P}\longleftrightarrow
\sum_{\lambda\in P}u_\lambda1_\lambda u'_\lambda).
\]

We denote by $\Mod(\dot{U})$ the category of 
finite-dimensional $\dot{U}$-modules $M$ with weight space decomposition
$
M=\bigoplus_{\lambda\in P}1_\lambda M$.
Then any $M\in\Mod(U_\ad)$ is regarded as an object of $\Mod(\dot{U})$ via the action of $\dot{U}$ on $M$ given by
\[
(u1_\lambda u')m
=u p^M_\lambda(u'm)
\qquad(u\in U^-, \;u'\in U^+),
\]
where $p^M_\lambda:M\to 1_\lambda M$ is the projection with respect to the weight space decomposition of $M$.
Moreover, this correspondence gives the equivalence of categories $\Mod(U_\ad)\cong\Mod(\dot{U})$ (see Lusztig \cite{Lbook}).
It follows that 
for each $\lambda\in P^+$ there exists uniquely (up to isomorphism) a finite-dimensional irreducible $\dot{U}$-module $L(\lambda)$ such that 
\[
L(\lambda)=\bigoplus_{\mu\in\lambda-Q^+}1_\mu L(\lambda),\qquad
\dim 1_\lambda L(\lambda)=1,\]
and that
any $M\in\Mod(\dot{U})$ is isomorphic to a direct sum of $L(\lambda)$'s for $\lambda\in P^+$.

\subsection{}
We denote by $V=V(\Delta)$ 
the associative algebra over $\BF$ generated by the elements
$t_\lambda\;(\lambda\in P)$, $x_i, y_i\;(i\in I)$ satisfying the fundamental relations
\begin{align*}
&t_0=1, \qquad t_\lambda t_\mu=t_{\lambda+\mu}\quad(\lambda, \mu\in P),
\\
&t_\lambda x_it_\lambda^{-1}=q_i^{(\lambda,\alpha_i^\vee)}x_i
\qquad(\lambda\in P,\;i\in I),
\\
&t_\lambda y_it_\lambda^{-1}=q_i^{(\lambda,\alpha_i^\vee)}y_i
\qquad(\lambda\in P,\;i\in I),
\\
&x_iy_j-y_jx_i=0
\qquad(i, j\in I),
\\
&\sum_{n=0}^{1-a_{ij}}(-1)^nx_i^{(1-a_{ij}-n)}x_jx_i^{(n)}=0
\qquad(i,j\in I,\,i\ne j),
\\
&\sum_{n=0}^{1-a_{ij}}(-1)^ny_i^{(1-a_{ij}-n)}y_jy_i^{(n)}=0
\qquad(i,j\in I,\,i\ne j),
\end{align*}
where $x_i^{(n)}=x_i^n/[n]_{q_i}!,\;y_i^{(n)}=y_i^n/[n]_{q_i}!\;\;(i\in I, n\in\BZ_{\geqq0})$.
We set $t_i=t_{\alpha_i}$ for $i\in I$.

Define subalgebras $V^0, V^+, V^-, V^{\geqq0}, V^{\leqq0}$ of $V$ by
\begin{align*}
&V^0=
\langle t_\lambda \mid
\lambda\in P\rangle,\qquad
V^+=
\langle x_i \mid
i\in I\rangle,\qquad
V^-=
\langle y_i \mid
i\in I\rangle,\\
&
V^{\geqq0}=
\langle t_\lambda,\; x_i \mid
\lambda\in P,\;i\in I\rangle,\qquad
V^{\leqq0}=
\langle t_\lambda,\; y_i \mid
\lambda\in P,\;i\in I\rangle.
\end{align*}
We have $V^0=\bigoplus_{\lambda\in P}\BF t_\lambda$, and the multiplication of $V$ induces isomorphisms
\begin{align*}
&
V^+\otimes V^0\otimes V^-\cong V^-\otimes V^0\otimes V^+\cong V,\\
&
V^+\otimes V^0\cong V^0\otimes V^+\cong V^{\geqq0},\qquad
V^-\otimes V^0\cong V^0\otimes V^-\cong V^{\leqq0}
\end{align*}
of vector spaces.
Moreover, we have algebra isomorphisms
\[
\jmath^+:V^+\to U^+\quad(x_i\mapsto e_i),
\qquad
\jmath^-:V^-\to U^-\quad(y_i\mapsto f_i).
\]
\begin{remark}
{\rm
$V$ is a $q$-analogue of the enveloping algebra of a certain solvable Lie subalgebra of $\Gg\oplus\Gg$, where $\Gg$ is a simple Lie algebra with root system $\Delta$ (see \ref{subsec:K} below).
}
\end{remark}

\subsection{}
The modified version $\dot{V}=\dot{V}(\Delta)$ is defined similarly to $\dot{U}$ as follows.
Denote by $V_{\ad}$ the $\BF$-subalgebra of $V$ generated by 
$t_\lambda\;(\lambda\in Q)$, $x_i, y_i\;(i\in I)$.
For $\gamma\in Q^+$ set 
$V_{\ad,\gamma}=\{v\in V_\ad\mid t_ivt_i^{-1}=q_i^{(\gamma,\alpha_i^\vee)}v\;(i\in I)\}$.
For $\lambda, \mu\in P$ we set 
\[
{}_\lambda \overline{V}_\mu=
V_\ad/
(
\sum_{i\in I}(t_i-q_i^{(\lambda,\alpha_i^\vee)})V_\ad
+
\sum_{i\in I}V_\ad(t_i-q_i^{(\mu,\alpha_i^\vee)})
)
\]
(note ${}_\lambda \overline{V}_\mu=0$ unless $\lambda-\mu\in Q^+$), and let
${}_\lambda \pi_\mu:V_\ad\to{}_\lambda \overline{V}_\mu$ be the natural map.
For $\lambda\in P$ set $1_\lambda={}_\lambda \pi_\lambda(1)$.
Set 
\[
\dot{V}=\bigoplus_{\lambda,\mu\in P}{}_\lambda \overline{V}_\mu.
\]
Then $\dot{V}$ is an associative algebra (without 1) by 
\[
{}_\lambda \pi_\mu(x){}_{\lambda'} \pi_{\mu'}(y)=
\begin{cases}
{}_\lambda \pi_{\mu'}(xy)
\quad&(\mu=\lambda')\\
0&(\mu\ne\lambda')
\end{cases}
\]
for $x\in V_{\ad,\lambda-\mu},\;y\in V_{\ad,\lambda'-\mu'}$.
Moreover, $\dot{V}$ is a $V_\ad$-bimodule by
\[
v\cdot
{}_\lambda \pi_\mu(x)
\cdot v'
={}_{\lambda+\gamma} \pi_{\mu-\gamma'}(vxv')
\qquad
(x\in V_{\ad,\lambda-\mu}, u\in V_{\ad,\gamma}, u'\in V_{\ad,\gamma'}).
\]
Then we have an isomorphism
\[
\bigoplus_{\lambda\in P} (V^-\otimes V^+)\cong\dot{V}
\qquad
((v_\lambda\otimes v'_\lambda)_{\lambda\in P}\longleftrightarrow
\sum_{\lambda\in P}v_\lambda v'_\lambda1_\lambda).
\]

Denote by $\Mod(\dot{V})$ (resp.\ $\Mod(V_\ad)$) the category of finite-dimensional $\dot{V}$-module (resp.\ $V_\ad$-module) with weight space decomposition.
Then we have a natural equivalence $\Mod(\dot{V})\cong\Mod(V_\ad)$ of categories.

\subsection{}
We denote by
\[
\tau:U_{\ad}^{\geqq0}\times U_{\ad}^{\leqq0}\to\BF
\]
the Drinfeld pairing.
It is a bilinear form uniquely determined by the properties
\begin{align}
&\tau(1,1)=1,\\
&\tau(x,y_1y_2)=(\tau\otimes\tau)(\Delta(x),y_1\otimes y_2)
&(x\in U_{\ad}^{\geqq0},\,y_1,y_2\in U_{\ad}^{\leqq0}),\\
&\tau(x_1x_2,y)=(\tau\otimes\tau)(x_2\otimes x_1,\Delta(y))
&(x_1, x_2\in U_{\ad}^{\geqq0},\,y\in U_{\ad}^{\leqq0}),\\
&\tau(k_i,k_j)=q_i^{-(\alpha_i^\vee,\alpha_j)}
&(i, j\in I),\\
&\tau(k_\lambda, f_i)=\tau(e_i,k_\lambda)=0
&(\lambda\in Q,\,i\in I),\\
&\tau(e_i,f_j)=\delta_{ij}/(q_i^{-1}-q_i)
&(i,j\in I).
\end{align}
We define a bilinear form
\[
\sigma:U\times \dot{V}\to\BF
\]
by
\begin{align*}
&\sigma(u_+k_\mu (Su_-), v_-v_+1_\lambda)
=
\tau(u_+,\jmath^-(v_-))\delta_{\lambda,\mu}\tau(\jmath^+(v_+), u_-)\\
&\qquad\qquad
(u_{\pm}\in U^{\pm},\; v_{\pm}\in V^{\pm},\;\lambda,\mu\in P).
\end{align*}

The following result is a consequence of  Gavarini \cite[Theorem 6.2]{Gav} (see also
\cite[Proposition 3.4]{TM}).

\begin{proposition}
\label{prop:sigma-inv} 
We have
\[
\sigma(u,vv')=
(\sigma\otimes\sigma)(\Delta(u),v\otimes v')
\qquad
(u\in U,\;v, v'\in \dot{V}).
\]
\end{proposition}

\subsection{}
For a Hopf algebra $H$ we define a left action of $H$ on $H$
by
\[
\ad(h)(h')=\sum_{j}h_{0j}h'(Sh_{1j})
\quad(h, h'\in H,\;\Delta(h)=\sum_jh_{0j}\otimes h_{1j}).
\]
We define a right action of $U_\ad$ on $\dot{U}$ by 
\[
x\cdot\widetilde{\ad}(u)=\sum_{j}(Su_{0j})xu_{1j}
\quad(x\in \dot{U},\; u\in U_\ad,\;\Delta(u)=\sum_ju_{0j}\otimes u_{1j}).\]

We set 
\[
{}^eU=\sum_{\lambda\in P}U^+k_{2\lambda}(SU^-)\subset U.
\]
Then ${}^eU$ is a subalgebra of $U$ satisfying
$
\ad(U)({}^eU)\subset {}^eU.
$
Define a bilinear form
\[
\omega:{}^eU\times \dot{U}\to\BF
\]
by
\begin{align*}
&\omega(u_+k_{2\mu} (Su_-), w_-1_{\lambda}(Sw_+))
=
\tau(u_+,w_-)
\delta_{\lambda,-\mu}
\tau(w_+, u_-)\\
&\qquad\qquad
(u_{\pm}, w_\pm\in U^{\pm},\; \lambda,\mu\in P).
\end{align*}

The following result is a consequence of \cite[Proposition 2.2.1]{T0}.
\begin{proposition}
\label{prop:omega-inv} 
We have
\[
\omega(\ad(u')(u),x)=
\omega(u,x\cdot\widetilde{\ad}(u'))
\qquad
(u\in{}^eU, u'\in U,\;x\in \dot{U}).
\]
\end{proposition}

Set
\[
{}^fU=\{u\in U\mid\dim\ad(U)(u)<\infty\}.
\]
Then  ${}^fU$ is a subalgebra of ${}^eU$ and we have
\[
{}^fU=\sum_{\lambda\in P^+}\ad(U)(k_{-2\lambda})
\]
(see \cite{JL}).

\subsection{}
We denote Lusztig's braid group action on $U$ by $T_i\;(i\in I)$.
Namely, $T_i:U\to U$ is 
the algebra automorphism given by
\begin{align*}
&T_i(k_\lambda)=k_{s_i(\lambda)}\quad(\lambda\in P),\\
&T_i(e_j)=
\begin{cases}
-f_ik_i&(i=j)\\
\sum_{r=0}^{-a_{ij}}(-1)^{-a_{ij}-r}q_i^{-r}e_i^{(-a_{ij}-r)}e_je_i^{(r)}
\qquad&(i\ne j),
\end{cases}\\
&T_i(f_j)=
\begin{cases}
-k_i^{-1}e_i&(i=j)\\
\sum_{r=0}^{-a_{ij}}(-1)^{r}q_i^{-a_{ij}-r}f_i^{(-a_{ij}-r)}f_jf_i^{(r)}
\qquad&(i\ne j).
\end{cases}
\end{align*}

We denote by $w_0$ the longest element of $W$.
We fix a reduced expression
$w_0=s_{i_1}\cdots s_{i_N}\;(i_1,\dots, i_N\in I)$ in the following.
For $j=1,\dots, N$ set $\beta_j=s_{i_1}\cdots s_{i_{j-1}}(\alpha_{i_j})$, and
\begin{align*}
&e_{\beta_j}=T_{i_1}\cdots T_{i_{j-1}}(e_{i_j}),
\qquad
f_{\beta_j}=T_{i_1}\cdots T_{i_{j-1}}(f_{i_j}),\\
&e_{\beta_j}^{(n)}=T_{i_1}\cdots T_{i_{j-1}}(e_{i_j}^{(n)}),
\qquad
f_{\beta_j}^{(n)}=T_{i_1}\cdots T_{i_{j-1}}(f_{i_j}^{(n)})\qquad(n\in\BZ_{\geqq0}).
\end{align*}
Then we have $\Delta^+=\{\beta_j\mid j=1,\dots, N\}$, and 
$e_\beta\in U^+,\; f_\beta\in U^-\;(\beta\in\Delta^+)$.
Moreover, the set 
$\{
e_{\beta_N}^{m_N}\cdots e_{\beta_1}^{m_1}
\mid m_j\in\BZ_{\geqq0}
\}$
(resp.\
$\{
f_{\beta_N}^{m_N}\cdots f_{\beta_1}^{m_1}
\mid m_j\in\BZ_{\geqq0}
\}$
)
is known to be a basis of $U^+$ (resp.\ $U^-$).
\subsection{}
We set $\CG=\CG(\Delta)=P/P_0$, where
\[
P_0=\{\lambda\in P
\mid
d_i(\lambda,\alpha_i^\vee)\in 2\BZ\;(i\in I)\}.
\]
Note that $\CG$ is a 2-elementary finite group.
For $\lambda\in P$ we denote by $\delta_\lambda\in\CG$ the element represented by $\lambda$.
We define an action of $\CG$ on the algebra $U$ by
\begin{align*}
\delta_\lambda(k_\mu)=k_\mu,\quad
\delta_\lambda(e_i)=(-1)^{d_i(\lambda,\alpha_i^\vee)}e_i,\quad
\delta_\lambda(f_i)=(-1)^{d_i(\lambda,\alpha_i^\vee)}f_i
\end{align*}
for $\lambda, \mu\in P, i\in I$.
We define an $\BF$-algebra structure of $\widetilde{U}=\widetilde{U}(\Delta)=U\otimes\BF[\CG]$ by
\[
(u\otimes \delta)(v\otimes \delta')=u \delta(v)\otimes \delta\delta'
\qquad(u, v\in U,\quad\delta, \delta'\in\CG).
\]
We will identify $U$ and $\BF[\CG]$ with the subalgebras 
$U\otimes 1$ and $1\otimes \BF[\CG]$
of $\widetilde{U}$ respectively.
We extend the $\CG$-action on $U$ to that on $\widetilde{U}$ by
$\delta(x)=\delta x\delta^{-1}\;(\delta\in\CG, x\in\widetilde{U})$.
Set 
\[
U^\CG=\{u\in U\mid \delta(u)=u\;(\delta\in\CG)\},\quad
\widetilde{U}^\CG=\{x\in \widetilde{U}\mid \delta(x)=x\;(\delta\in\CG)\}.
\]
Then we see easily that 
$
\widetilde{U}^\CG={U}^\CG\BF[\CG].
$
\subsection{}
Let $\theta$ be the automorphism of the field $\BF$ sending $q$ to $-q$.
For an $\BF$-algebra $R$ we
denote by ${}^\theta R$ the $\BF$-algebra obtained by twisting the $\BF$-module structure of $R$ via $\theta$.
Namely,  ${}^\theta R$  is isomorphic to $R$ as a ring via the correspondence 
$R\ni x\leftrightarrow{}^\theta x\in{}^\theta R$, and the $\BF$-module structure is given by 
$c\,{}^\theta x={}^\theta(\theta(c)x)\;(c\in\BF, x\in R)$.

Now we are going to define an embedding of ${}^\theta U$ into $\widetilde{U}$ following \cite{KKO}.
We can take a subset $J$ of $I$ such that for $i, j\in I$ with $i\ne j$ we have
\[
d_i(\alpha_i^\vee,\alpha_j)\not\in 2\BZ
\Longrightarrow
|\{i, j\}\cap J|=1.
\]
For $i\in I$ set
\[
\varphi_i=
\begin{cases}
\delta_{\alpha_i}\qquad&(i\in J)\\
1&(i\not\in J),
\end{cases}
\qquad
\psi_i=(-1)^{d_i}\varphi_i\delta_{\alpha_i}.
\]
For $\gamma=\sum_{i\in I}m_i\alpha_i\in Q$ we further set
\[
\varphi_\gamma=\prod_{i\in I}\varphi_i^{m_i},\qquad
\psi_\gamma=\prod_{i\in I}\psi_i^{m_i}.
\]

\begin{proposition}[\cite{KKO}]
\label{prop:kashiwara1}
An embedding ${}^\theta U\to\widetilde{U}$ of $\BF$-algebras is given by
\[
{}^\theta k_\lambda\mapsto k_\lambda \delta_\lambda,\qquad
{}^\theta e_i\mapsto e_i\varphi_i,\qquad
{}^\theta f_i\mapsto f_i\psi_i.
\]
\end{proposition}
\begin{remark}
{\rm
In \cite{KKO} Kashiwara-Kang-Oh established using Proposition \ref{prop:kashiwara1} 
the equivalence $\Mod(U)\cong\Mod({}^\theta U)$, where 
$\Mod(U)$ (resp.\ $\Mod({}^\theta U)$) denotes the category of $U$-modules (resp.\ ${}^\theta U$-modules) with weight space decompositions (see also Andersen \cite{Andersen}).
}
\end{remark}
We will identify ${}^\theta U$ with a subalgebra of $\widetilde{U}$.
We can easily check the following.
\begin{lemma}
\label{lem:kashiwara0}
\begin{itemize}
\item[(i)]
The multiplication of $\widetilde{U}$ gives an isomorphism 
${}^\theta U\otimes \BF[\CG]\cong\widetilde{U}$ of $\BF$-modules.
\item[(ii)]
For any $\delta\in\CG$ and ${}^\theta u\in{}^\theta U$ we have 
$
\delta \,{}^\theta u \delta^{-1}={}^\theta(\delta(u))
$.
\end{itemize}
\end{lemma}
\begin{proposition}
\label{prop:kashiwara2}
For any $\beta\in\Delta^+$ we have 
\[
{}^\theta e_\beta=\pm e_\beta\varphi_\beta,\qquad
{}^\theta(Sf_\beta)=\pm (Sf_\beta)\varphi_\beta.
\]
\end{proposition}
\begin{proof}
For $i\in I$ define ${}^\theta T_i:{}^\theta U\to{}^\theta U$ by
${}^\theta T_i({}^\theta u)={}^\theta(T_i(u))\;(u\in U)$.
For $\gamma\in Q^+$ set
\[
U^+_\gamma=\{x\in U^+\mid k_ixk_i^{-1}=q_i^{(\alpha_i^\vee,\gamma)}x\;(i\in I)\}.
\]
For $i\in I$ we set 
\[
\varpi_i=
\begin{cases}
-1\qquad&(i\in J),
\\
1\qquad&(i\notin J).
\end{cases}
\]
In order to show the statement for $e_\beta$, it is sufficient to show that for $\gamma\in Q^+$ and $i\in I$ there exists $c_{i,\gamma}\in\{\pm1\}$ satisfying
\begin{equation}
\label{eq:p1}
{}^\theta T_i(x\varphi_\gamma)=c_{i,\gamma}T_i(x)\varphi_{s_i\gamma}
\qquad(x\in U^+_\gamma).
\end{equation}
We first note that for $i, j\in I$ we have
\begin{equation}
\label{eq:p2}
{}^\theta T_i(e_j\varphi_j)=c_{i,j}T_i(e_j)\varphi_{s_i\alpha_j},
\end{equation}
where
\[
c_{ij}=
(-1)^{d_ia_{ij}(a_{ij}+1)/2}
\varpi_i^{d_ia_{ij}}\in\{\pm1\}.
\]
The verification of \eqref{eq:p2} in the case $i=j$ is easy.
In the case $i\ne j$ one needs some case by case calculation according to the relative position of $\alpha_i$ and $\alpha_j$.
Details are omitted.
Now let us show \eqref{eq:p1} using \eqref{eq:p2}.
For $j_1,\dots, j_r\in I$ with $\sum_p\alpha_{j_p}=\gamma$ we have
\[
(e_{j_1}\varphi_{j_1})\cdots(e_{j_r}\varphi_{j_r})
=
A
e_{j_1}\cdots e_{j_r}\varphi_{\gamma}
\]
where $A=\prod_{p=1}^{r-1}\varpi_{j_p}^{d_{j_p}(\alpha^\vee_{j_p},\alpha_{j_{p+1}}+\cdots+\alpha_{j_r})}$.
Hence we have
\begin{align*}
{}^\theta{T}_i((e_{j_1}\cdots e_{j_r})\varphi_{\gamma})
=&
A
{}^\theta{T}_i((e_{j_1}\varphi_{j_1})\cdots(e_{j_r}\varphi_{j_r}))\\
=&
A
\left(\prod_pc_{ij_p}\right)
((T_ie_{j_1})\varphi_{\alpha_{j_1}}\varphi_{\alpha_i}^{-a_{ij_1}})
\cdots
((T_ie_{j_r})\varphi_{\alpha_{j_r}}\varphi_{\alpha_i}^{-a_{ij_r}})
\\
=&
AA'
\left(\prod_pc_{ij_p}\right)
T_i(e_{j_1}\cdots e_{j_r})\varphi_{s_i\gamma},
\end{align*}
with
\begin{align*}
A'=&
\left(
\prod_{p=1}^{r-1}
\varpi_{j_p}^{d_{j_p}(\alpha^\vee_{j_p},s_i(\alpha_{j_{p+1}}+\cdots+\alpha_{j_r}))}
\right)
\left(
\prod_{p=1}^{r-1}
\varpi_{i}^{-d_ia_{ij_p}(\alpha^\vee_{i},s_i(\alpha_{j_{p+1}}+\cdots+\alpha_{j_r}))}
\right)
\\
=&
A
\left(
\prod_{p=1}^{r-1}
\varpi_{j_p}^{d_{j_p}(\alpha^\vee_{j_p},\alpha_i)(\alpha_i^\vee,\alpha_{j_{p+1}}+\cdots+\alpha_{j_r})}
\right)
\left(
\prod_{p=1}^{r-1}
\varpi_{i}^{d_{j_p}(\alpha^\vee_{j_p},\alpha_i)(\alpha^\vee_{i},\alpha_{j_{p+1}}+\cdots+\alpha_{j_r})}
\right)
\\
=&
A
\prod_{p=1}^{r-1}
(\varpi_i\varpi_{j_p})^{d_{j_p}(\alpha^\vee_{j_p},\alpha_i)(\alpha_i^\vee,\alpha_{j_{p+1}}+\cdots+\alpha_{j_r})}
\\
=&
A
\prod_{p=1}^{r-1}
(-1)^{d_i(\alpha_i^\vee,\alpha_{j_p})(\alpha_i^\vee,\alpha_{j_{p+1}}+\cdots+\alpha_{j_r})}
\\
=&
A
\prod_{1\leqq p<p'\leqq r}
(-1)^{d_i(\alpha_i^\vee,\alpha_{j_p})(\alpha_i^\vee,\alpha_{j_{p'}})}
.
\end{align*}
We thus obtain \eqref{eq:p1}, where
\[
c_{i,\gamma}
=\left(\prod_{p}c_{ij_p}\right)
\left(\prod_{p<p'}(-1)^{d_i(\alpha_i^\vee,\alpha_{j_p})(\alpha_i^\vee,\alpha_{j_{p'}})}\right).
\]
The proof of the assertion for $f_\beta$ is similar.
\end{proof}
\subsection{}
Set $\CH=\CH(\Delta)=Q^\vee/2Q^\vee$.
For $\nu\in Q^\vee$ we denote by $\gamma_\nu$ the element of $\CH$ represented by $\nu$.
Define an action of $\CH$ on
the  $\BF$-algebra $U^0=\bigoplus_{\lambda\in P}\BF k_\lambda\cong\BF[P]$  by
\[
{\gamma}_{\nu}\cdot k_\lambda
=(-1)^{(\nu,\lambda)}
k_\lambda
\qquad
(\nu\in Q^\vee,\; \lambda\in P).
\]
We can extend this $\CH$-action on $U^0$ to
that on the algebra
$U\cong 
U^+\otimes
SU^-
\otimes
U^0$
by
\[
\gamma\cdot (ut)=u(\gamma\cdot t)
\qquad(\gamma\in\CH, u\in U^+(SU^-), t\in U^0).
\]
Since this action commutes with that of $\CG$, we get an action of 
$\CG\times \CH$ on $U$.

\subsection{}
Set $\BA=\BQ[q^{\pm1}]$.
Following De Concini-Procesi \cite{DP} we define $U_\BA$ to be the smallest $\BA$-subalgebra of $U$ that contains $k_\lambda\;(\lambda\in P),\; (q_i-q_i^{-1})e_i, \; (q_i-q_i^{-1})f_i\;(i\in I)$ and is stable under the action of $T_i\;(i\in I)$.
It is a Hopf algebra over $\BA$.
Set 
\[
U_\BA^0=U_\BA\cap U^0,\quad
U_\BA^\pm=U_\BA\cap U^\pm,\quad
U_\BA^{\geqq0}=U_\BA\cap U^{\geqq0},\quad
U_\BA^{\leqq0}=U_\BA\cap U^{\leqq0}.
\]
Then we have $U_\BA^0=\bigoplus_{\lambda\in P}\BA k_\lambda$, and the multiplication of $U_\BA$ induces isomorphisms
\begin{align*}
&
U_\BA^+\otimes U_\BA^0\otimes U_\BA^-\cong U_\BA^-\otimes U_\BA^0\otimes U_\BA^+\cong U_\BA,\\
&
U_\BA^+\otimes U_\BA^0\cong U_\BA^0\otimes U_\BA^+\cong U_\BA^{\geqq0},\qquad
U_\BA^-\otimes U_\BA^0\cong U_\BA^0\otimes U_\BA^-\cong U_\BA^{\leqq0}
\end{align*}
of $\BA$-modules.
For $\beta\in \Delta^+$ we define $a_\beta\in U_\BA^+,\;b_\beta\in U_\BA^-$ by
\[
a_{\beta}=(q_{\beta}-q_{\beta}^{-1})e_{\beta},\qquad
b_{\beta}=(q_{\beta}-q_{\beta}^{-1})f_{\beta}.
\]
Then 
$\{
a_{\beta_N}^{m_N}\cdots a_{\beta_1}^{m_1}
\mid m_j\in\BZ_{\geqq0}\}$ 
(resp.\
$\{
b_{\beta_N}^{m_N}\cdots b_{\beta_1}^{m_1}
\mid m_j\in\BZ_{\geqq0}\}$
)
is a free $\BA$-basis of $U_\BA^+$
(resp.\ $U_\BA^-$).

Set
\[
U_{\ad,\BA}=U_\BA\cap U_\ad,\qquad
U^\flat_{\ad,\BA}=U_\BA\cap U^\flat_\ad\quad(\flat=0,\geqq0, \leqq0).
\]
Then we have $U_{\ad,\BA}^0=\bigoplus_{\lambda\in Q}\BA k_\lambda$, and
\begin{align*}
&
U_\BA^+\otimes U_{\ad,\BA}^0\otimes U_\BA^-\cong U_\BA^-\otimes U_{\ad,\BA}^0\otimes U_\BA^+\cong U_{\ad,\BA},\\
&
U_\BA^+\otimes U_{\ad,\BA}^0\cong U_{\ad,\BA}^0\otimes U_\BA^+\cong U_{\ad,\BA}^{\geqq0},\qquad
U_\BA^-\otimes U_{\ad,\BA}^0\cong U_{\ad,\BA}^0\otimes U_\BA^-\cong U_{\ad,\BA}^{\leqq0}.
\end{align*}

Denote by $U_{\ad,\BA}^{L}$ the $\BA$-subalgebra of $U$ generated by
the elements $\{e_i^{(n)}, f_i^{(n)}, k_\lambda\mid i\in I, n\in\BZ_{\geqq0}, \lambda\in Q\}$, and set
\[
U_{\ad,\BA}^{L,\flat}=U_{\ad,\BA}^{L}\cap U_\ad^\flat\quad
(\flat=0,\geqq0, \leqq0),\qquad
U_{\BA}^{L,\pm}=U_{\ad,\BA}^{L}\cap U^\pm.
\]
Then we have
\begin{align*}
&
U_\BA^{L,+}\otimes U_{\ad,\BA}^{L,0}\otimes U_\BA^{L,-}\cong U_\BA^{L,-}\otimes U_{\ad,\BA}^{L,0}\otimes U_\BA^{L,+}\cong U^L_{\ad,\BA},\\
&
U_\BA^{L,+}\otimes U_{\ad,\BA}^{L,0}\cong U_{\ad,\BA}^{L,0}\otimes U_\BA^{L,+}\cong U_{\ad,\BA}^{L,\geqq0},\\
&
U_\BA^{L,-}\otimes U_{\ad,\BA}^{L,0}\cong U_{\ad,\BA}^{L,0}\otimes U_\BA^{L,-}\cong U_{\ad,\BA}^{L,\leqq0}.
\end{align*}
Moreover, $U_{\ad,\BA}^{L,0}$ is generated by the elements of the form
$k_\lambda\; (\lambda\in Q)$, 
\[
\begin{bmatrix}
k_i\\m
\end{bmatrix}
=\prod_{s=0}^{m-1}
\frac{q_i^{-s}k_i-q_i^sk_i^{-1}}{q_i^{s+1}-q_i^{-s-1}}
\qquad(i\in I, m\geqq0),
\]
and
$\{
e_{\beta_N}^{(m_N)}\cdots e_{\beta_1}^{(m_1)}
\mid m_j\in\BZ_{\geqq0}\}$ 
(resp.\
$\{
f_{\beta_N}^{(m_N)}\cdots f_{\beta_1}^{(m_1)}
\mid m_j\in\BZ_{\geqq0}\}$
)
is a free $\BA$-basis of $U_\BA^{L,+}$
(resp.\ $U_\BA^{L,-}$).

We define $\dot{U}_\BA$ to be the $\BA$-subalgebra of $\dot{U}$ consisting of elements of the form
\[
\sum_{\lambda\in P}u_\lambda1_\lambda u'_\lambda
\qquad
(u_\lambda\in U_\BA^{L,-}, \;u'_\lambda\in U_\BA^{L,+}).
\]
For $\lambda\in P^+$ we define an $\BA$-form $L_\BA(\lambda)$ of $L(\lambda)$ by
\[
L_\BA(\lambda)=\dot{U}_\BA v
\qquad(1_\lambda L(\lambda)=\BF v).
\]

We define $\dot{V}_\BA$ to be the $\BA$-subalgebra of $\dot{V}$ consisting of elements of the form
\[
\sum_{\lambda\in P}v_\lambda v'_\lambda1_\lambda
\qquad
(v_\lambda\in (\jmath^-)^{-1}(U_\BA^{L,-}), \;v'_\lambda\in (\jmath^+)^{-1}(U_\BA^{L,+})).
\]
We set
\[
{}^eU_\BA={}^eU\cap U_\BA,\qquad
{}^fU_\BA={}^fU\cap U_\BA.
\]

By \cite{KR}, \cite{KT}, \cite{LS}
we have
\begin{align}
\label{eq:Drinfeld-value}
&\tau(
e_{\beta_N}^{(m_N)}\cdots e_{\beta_1}^{(m_1)},
b_{\beta_N}^{n_N}\cdots b_{\beta_1}^{n_1})
=
\tau(
a_{\beta_N}^{m_N}\cdots a_{\beta_1}^{m_1},
f_{\beta_N}^{(n_N)}\cdots f_{\beta_1}^{(n_1)})\\
\nonumber
=&
\prod_{s=1}^N
\delta_{m_s,n_s}(-1)^{m_s}q_{\beta_s}^{m_s(m_s-1)/2},
\end{align}
and hence $\tau$ induces bilinear forms
\[
\tau^{\emptyset,L}_\BA:
U_{\ad,\BA}^{\geqq0}\times U^{L,\leqq0}_{\ad,\BA}\to\BA,\qquad
\tau^{L,\emptyset}_\BA:
U_{\ad,\BA}^{L,\geqq0}\times U^{\leqq0}_{\ad,\BA}\to\BA.
\]
It follows that $\sigma$ and $\omega$ also induce perfect bilinear forms
\begin{align*}
\sigma_\BA:U_\BA\times \dot{V}_\BA\to\BA,\qquad
\omega_\BA:{}^eU_\BA\times \dot{U}_\BA\to\BA.
\end{align*}

Set $\widetilde{U}_\BA=U_\BA\otimes\BA[\CG]$.
It is an $\BA$-subalgebra of $\widetilde{U}$.
We also have an obvious $\BA$-form ${}^\theta U_\BA$ of ${}^\theta U$.
By Proposition \ref{prop:kashiwara2} the embedding
${}^\theta U\subset\widetilde{U}$
induces
${}^\theta U_\BA\to \widetilde{U}_\BA$.

\subsection{}
Let $z\in\BC^\times$, and set
\begin{equation}
\label{eq:kappa-beta}
z_\beta=z^{d_\beta}\quad(\beta\in\Delta),\qquad
z_i=z_{\alpha_i}\quad(i\in I).
\end{equation}
Set
\[
U_z=U_z(\Delta)=\BC\otimes_\BA U_\BA,
\]
where $\BA\to\BC$ is given by $q\mapsto z$.
We also set
\begin{align*}
&U^\flat_z=\BC\otimes_\BA U^\flat_\BA\qquad
(\flat=\emptyset, +, -, 0, \geqq0, \leqq0),\\
&U_{\ad,z}^{\flat}=\BC\otimes_\BA U_{\ad,\BA}^\flat,\quad
U_{\ad,z}^{L,\flat}=\BC\otimes_\BA U_{\ad,\BA}^\flat\qquad
(\flat=\emptyset, 0, \geqq0, \leqq0),\\
&U^{L,\pm}_z=\BC\otimes_\BA U^{L,\pm}_\BA,
\\
&
\dot{U}_z=\BC\otimes_\BA \dot{U}_\BA, \qquad
\dot{V}_z=\BC\otimes_\BA \dot{V}_\BA,\\
&{}^eU_z=\BC\otimes_\BA{}^eU_\BA,\qquad
{}^fU_z=\BC\otimes_\BA{}^fU_\BA.
\end{align*}
Then we have
\[
U_z\cong U_z^-\otimes U_z^0\otimes U_z^+,\qquad
\dot{U}_z\cong \bigoplus_{\lambda\in P}U_z^{L,-}1_\lambda U_z^{L,+}.
\]
Since $U_\BA$ is a free $\BA$-module, we have
$
{}^fU_z\subset{}^eU_z\subset U_z.
$

We denote by $\Mod(\dot{U}_z)$ the category of 
finite-dimensional $\dot{U}_z$-modules $M$ with weight space decomposition
$
M=\bigoplus_{\lambda\in P}1_\lambda M$.
For $\lambda\in P^+$ we define $L_z(\lambda)\in\Mod(\dot{U}_z)$ by
\[
L_z(\lambda)=\BC\otimes_\BA L_\BA(\lambda).
\]

Note that $\tau^{\emptyset,L}_\BA$, $\tau^{L,\emptyset}_\BA$, $\sigma_\BA$ and $\omega_\BA$ induce bilinear forms
\begin{align*}
&\tau^{\emptyset,L}_z:
U_{\ad,z}^{\geqq0}\times U^{L,\leqq0}_{\ad,z}\to\BC,\qquad
\tau^{L,\emptyset}_z:
U_{\ad,z}^{L,\geqq0}\times U^{\leqq0}_{\ad,z}\to\BC,
\\
&\sigma_z:U_z\times \dot{V}_z\to\BC,\qquad
\omega_z:{}^eU_z\times \dot{U}_{z}\to\BC.
\end{align*}
By \eqref{eq:Drinfeld-value}
$\tau^{\emptyset,L}_z|_{
U_{z}^{+}\times U^{L,-}_{z}}$, 
$\tau^{L,\emptyset}_z|
{U_{z}^{L,+}\times U^{-}_{z}}$, $\sigma_z$, $\omega_z$ are perfect.

Set $\widetilde{U}_z=\BC\otimes_\BA\widetilde{U}_\BA=U_z\otimes\BC[\CG]$.
Then we have a natural embedding ${U}_z\subset\widetilde{U}_z$, which is compatible with the $\CG$-actions.
Note that the embedding ${}^\theta U_\BA\to \widetilde{U}_\BA$ also induces an embedding $U_{-z}\subset \widetilde{U}_z$, which is compatible with $\CG$-actions.
Hence setting
\begin{align*}
U_{z}^\CG&=
\{u\in U_{z}\mid \delta(u)=u\;(\delta\in\CG)\},\\
\widetilde{U}_z^\CG&=
\{x\in \widetilde{U}_z\mid \delta(x)=x\;(\delta\in\CG)\},
\end{align*}
we obtain embeddings
\[
U_{-z}\subset \widetilde{U}_z\supset U_z,\qquad
U_{-z}^\CG\subset \widetilde{U}_z^\CG\supset U_z^\CG.
\]
We denote by 
$
\widetilde{\Xi}_z:U_{-z}\to U_z
$
the restriction of the linear map $\widetilde{U}_z\to U_z$, 
which sends
$u\delta_\lambda$ for $u\in U_z,\; \lambda\in P$ to $u$.

\begin{proposition}
\label{prop:Xi}
The linear map
$\widetilde{\Xi}_z$ induces an isomorphism
\begin{equation}
\Xi_z:U_{-z}^\CG\to U_{z}^\CG
\end{equation}
of $\BC$-algebras, which is compatible with the $\CH$-actions.
\end{proposition}
\begin{proof}
Since $\widetilde{\Xi}_z$ is a linear isomorphism compatible with $\CG$-actions,
it induces a
 linear isomorphism
$
\Xi_z:U_{-z}^\CG\to U_{z}^\CG
$.
Note that
$U_{-z}^\CG\subset \widetilde{U}_z^\CG=U_z^\CG\BC[\CG]$.
For $u, u'\in  U_z^\CG$, $\delta, \delta'\in\CG$
we have
\[
\widetilde{\Xi}_z((u\delta)(u'\delta'))
=\widetilde{\Xi}_z(uu'\delta\delta')
=uu'
=\widetilde{\Xi}_z(u\delta)
\widetilde{\Xi}_z(u'\delta').
\]
Hence $\widetilde{\Xi}_z|U_z^\CG\BC[\CG]:U_z^\CG\BC[\CG]\to U_z^\CG$ is an algebra homomorphism.
It follows that its restriction $
\Xi_z:U_{-z}^\CG\to U_{z}^\CG
$
is also an algebra homomorphism.
The remaining statement about the action of $\CH$ is obvious.
\end{proof}

\subsection{}
\label{subsec:K}
Let  $G=G(\Delta)$ be a connected, simply-connected semisimple algebraic group over $\BC$ with root system $\Delta$.
Take a maximal torus $H=H(\Delta)$ of $G$ and Borel subgroups $B^+, B^-$ of $G$ such that $B^+\cap B^-=H$.
Set $N^\pm=[B^\pm,B^\pm]$, and define a closed subgroup $K=K(\Delta)$ of $B^+\times B^-$ by
\[
K=\{(gh,g'h^{-1})\mid h\in H, g\in N^+, g'\in N^-\}.
\]
Setting
\[
K^0=\{(h,h^{-1})\mid h\in H\}\cong H,
\]
\[
K^+=
\{(g,1)\mid g\in N^+\},
\qquad
K^-=\{(1,g')\mid g'\in N^-\}\cong N^-,
\]
we obtain an isomorphism
\[
K^+\times K^0\times K^-\to K
\qquad((a,b,c)\mapsto abc)
\]
of algebraic varieties.
Denote by $\Gg$ the Lie algebra of $G$.
It is generated by the elements $\overline{h}_i, \overline{e}_i, \overline{f}_i\;(i\in I)$ satisfying the fundamental relations
\begin{align*}
&[\overline{h}_i,\overline{h}_j]=0\qquad &(i, j\in I),
\\
&[\overline{h}_i,\overline{e}_j]=a_{ij}\overline{e}_j,\qquad
[\overline{h}_i,\overline{f}_j]=-a_{ij}\overline{f}_j\qquad &(i, j\in I),
\\
&[\overline{e}_i,\overline{f}_j]=\delta_{ij}
\overline{h}_i
\qquad&(i, j\in I),
\\
&\ad(\overline{e}_i)^{1-a_{ij}}(\overline{e}_j)=
\ad(\overline{f}_i)^{1-a_{ij}}(\overline{f}_j)
=0
\qquad&(i,j\in I,\,i\ne j).
\end{align*}
Then the Lie algebra $\Gk$ of $K$ is the subalgebra of $\Gg\oplus\Gg$ generated by elements 
$\overline{t}_i=(\overline{h}_i,-\overline{h_i})$, 
$\overline{x}_i=(\overline{e}_i,0)$, 
$\overline{y}_i=(0,\overline{f_i})$\;($i\in I$).
Those generators satisfy the fundamental relations
\begin{align*}
&[\overline{t}_i,\overline{t}_j]=0\qquad &(i, j\in I),
\\
&[\overline{t}_i,\overline{x}_j]=a_{ij}\overline{x}_j,\qquad
[\overline{t}_i,\overline{y}_j]=a_{ij}\overline{y}_j\qquad &(i, j\in I),
\\
&[\overline{x}_i,\overline{y}_j]=0\qquad&(i, j\in I),
\\
&\ad(\overline{x}_i)^{1-a_{ij}}(\overline{x}_j)=
\ad(\overline{y}_i)^{1-a_{ij}}(\overline{y}_j)
=0
\qquad&(i,j\in I,\,i\ne j).
\end{align*}
Let $U(\Gk)$ be the enveloping algebra of $\Gk$.
We can define the modified version $\dot{U}(\Gk)$ of $U(\Gk)$ similarly to
$\dot{V}$ as follows.
For $\gamma\in Q^+$ set 
$U(\Gk)_{\gamma}=
\{u\in U(\Gk)\mid 
[\overline{t}_i,u]={(\gamma,\alpha_i^\vee)}u\;(i\in I)\}$.
For $\lambda, \mu\in P$ we set 
\[
{}_\lambda \overline{U}(\Gk)_\mu=
{U}(\Gk)/
(
\sum_{i\in I}(\overline{t}_i-{(\lambda,\alpha_i^\vee)}){U}(\Gk)
+
\sum_{i\in I}{U}(\Gk)(\overline{t}_i-{(\mu,\alpha_i^\vee)})
),
\]
and let
${}_\lambda \pi^1_\mu:U(\Gk)\to
{}_\lambda \overline{U}(\Gk)_\mu$ be the natural map.
Set 
\[
\dot{U}(\Gk)=\bigoplus_{\lambda,\mu\in P}
{}_\lambda \overline{U}(\Gk)_\mu.
\]
Then $\dot{U}(\Gk)$ is an associative algebra (without 1) by 
\[
{}_\lambda \pi^1_\mu(x){}_{\lambda'} \pi^1_{\mu'}(y)=
\begin{cases}
{}_\lambda \pi^1_{\mu'}(xy)
\quad&(\mu=\lambda')\\
0&(\mu\ne\lambda')
\end{cases}
\]
for $x\in U(\Gk)_{\lambda-\mu},\;y\in U(\Gk)_{\lambda'-\mu'}$.
It is easily seen that we have $\dot{V}_1\cong\dot{U}(\Gk)$.

Regard $\BC[K]$ as a $U(\Gk)$-module by differentiating the $K$-action 
\[
(kf)(k')=f(k'k)\;(k, k'\in\BC[K],\,f\in\BC[K])
\]
on $\BC[K]$.
Since $\BC[K]$ is a sum of finite dimensional $U(\Gk)$-submodules with weight space decomposition, we obtain a natural action of $\dot{U}(\Gk)$ on $\BC[K]$.
Cosider the bilinear form
\[
\overline{\sigma}:\BC[K]\times \dot{U}(\Gk)\to\BC
\qquad(
\overline{\sigma}(f,x)=(xf)(1)),
\]
By Proposition \ref{prop:sigma-inv} and $K\cong K^+\times K^0\times K^-$, 
we see easily that an isomorphism
\begin{equation}
\label{eq:U1K}
\Upsilon:U_1\to\BC[K]
\end{equation}
of coalgebras is given by
\[
\overline{\sigma}(\Upsilon(u),x)
=
\sigma_1(u,x)
\qquad(u\in U_1, \;x\in \dot{V}_1=\dot{U}(\Gk)).
\]
Since $U_1$ and $\BC[K]$ are commutative, it is easily seen  that \eqref{eq:U1K} is an isomorphism of Hopf algebras (see \cite{DP}, \cite{Gav}, \cite{TM}).

\section{Harish-Chandra center}
\label{sec:Har}
\subsection{}
For a ring $R$ we denote its center by $Z(R)$.

Consider the composite of
\[
Z(U)\hookrightarrow U\cong U^-\otimes U^0\otimes U^+\xrightarrow{\varepsilon\otimes1\otimes\varepsilon} U^0
\cong\BF[P],
\]
where $\BF[P]=\bigoplus_{\lambda\in P}\BF e(\lambda)$ is the group algebra of $P$, and the isomorphism
$U^0\cong\BF[P]$ is given by $k_\lambda\leftrightarrow e(\lambda)$.
By \cite{DK}, \cite{JL}, \cite{T0} this linear map $Z(U)\to\BF[P]$ is an injective algebra homomorphism whose image coincides with
\[
\BF[2P]^{W\circ}=\{x\in\BF[2P]\mid w\circ x=x\;(w\in W)\},
\]
where the action of $W$ on $\BF[2P]$ is given by
\[
w\circ e(2\lambda)=q^{(w\lambda-\lambda,2\tilde{\rho})} e(2w\lambda)\qquad
(w\in W,\;\lambda\in P).
\]
Hence we have an isomorphism
\begin{equation}
\iota:Z(U)\to\BF[2P]^{W\circ}.
\end{equation}

We recall here a description of $Z(U)$ in terms of the characters of finite-dimensional ${U}$-modules.
For $M\in\Mod(\dot{U})$ 
we define $\tilde{t}_M\in\dot{U}^*$ by 
\[
\langle \tilde{t}_M,x\rangle
=\trace(xk_{2\rho},M)\qquad(x\in\dot{U}).
\]
Then there exists uniquely an element $t_M\in {}^eU$ satisfying 
\[
\omega(t_M,x)=\langle\tilde{t}_M,x\rangle\qquad
(x\in \dot{U}).
\]
More explicitly, we have
\begin{align*}
t_M
=
\sum_{\lambda\in2P,\sum_{j=1}^N(m_j-m'_j)\beta_j=0}
c_{\lambda,\{m_j\}_{j=1}^N,\{m'_j\}_{j=1}^N}
a_{\beta_1}^{m_1}\cdots a_{\beta_N}^{m_N}
k_\lambda
S(b_{\beta_N}^{m'_N}\cdots b_{\beta_1}^{m'_1}),
\end{align*}
where
\begin{align*}
&c_{\lambda,\{m_j\}_{j=1}^N,\{m'_j\}_{j=1}^N}\\
=&
\prod_{j=1}^N
(-1)^{m_j+m'_j}q_{\beta_j}^{-m_j(m_j-1)/2-m'_j(m'_j-1)/2}\\
&
\times
\trace\left(\pi\left\{
f_{\beta_1}^{(m_1)}\cdots f_{\beta_N}^{(m_N)}
1_{-\frac{\lambda}2}
S(e_{\beta_N}^{(m'_N)}\cdots e_{\beta_1}^{(m'_1)})k_{2\rho}
\right\},1_{-\frac{\lambda}2-\sum_jm'_j\beta_j}M
\right).
\end{align*}

We can show $t_M\in Z(U)$ using $k_{2\rho}^{-1}uk_{2\rho}=S^2u\;(u\in U)$,
$Z(U)=\{v\in U\mid\ad(u)(v)=\varepsilon(u)v\;(u\in U)\}$, and
Proposition \ref{prop:omega-inv} (see \cite{T0}).
We have 
\[
\iota(t_M)=
\sum_{\lambda\in P}
(\dim 1_\lambda M)
q^{(\lambda,2\tilde{\rho})}e({-2\lambda}).
\]
\begin{proposition}
\label{prop:ZU}
\begin{itemize}
\item[(i)]
$Z(U)\subset U^\CG$.
\item[(ii)]
We have
\[
Z({}^\theta U)=Z(\widetilde{U})=Z(U)
\]
as subalgebras of $\widetilde{U}$.
Moreover, the composite of 
\[
\BF[2P]^{W\circ}\cong Z(U)=
Z({}^\theta U)\cong {}^\theta Z(U)\cong
{}^\theta (\BF[2P]^{W\circ})
\]
is induced by the $\BF$-linear isomorphism
\[
\BF[2P]\ni e(2\lambda)\mapsto {}^\theta e(2\lambda)\in {}^\theta \BF[2P].
\]
\end{itemize}
\end{proposition}
\begin{proof}
(i) Let $\delta\in\CG$.
Since $\delta$ acts on $U$ as an algebra automorphism, we have $\delta(Z(U))=Z(U)$.
It is easily seen from the definition of $\delta$ that 
$\iota(\delta(z))=\iota(z)$ for any $z\in Z(U)$.
Hence $\delta$ acts as identity on $Z(U)$.

(ii) By (i) we have $Z(U)\subset Z(\widetilde{U})$.
Let us show $Z(U)\supset Z(\widetilde{U})$.
Let $z=\sum_{\delta\in\CG}u_\delta\delta\in Z(\widetilde{U})$, where $u_\delta\in U$.
By $uz=zu$ for $u\in U$ we have $uu_\delta=u_\delta\delta(u)$.
By considering the corresponding identity in the associated graded algebra $\Gr\, U$ introduced in \cite{DKP2} we see easily that 
$u_\delta=0$ for $\delta\ne1$.
Hence $z\in Z(U)$.
The proof of $Z({}^\theta U)=Z(\widetilde{U})$ is similar.
The remaining statement is a consequence of ${}^\theta k_{2\lambda}=k_{2\lambda}$ for $\lambda\in P$.
\end{proof}

\subsection{}
By $Z(U_\BA)=U_\BA\cap Z(U)$ $\iota$ induces an injective algebra homomorphism
\[
\iota_\BA:Z(U_\BA)\to\BA[2P]^{W\circ}
\]
\begin{proposition}
\label{prop:HCA}
$\iota_\BA$ is an isomorphism of $\BA$-algebras.
\end{proposition}
\begin{proof}
For $\lambda\in P^+$ we have $t_{L(\lambda)}\in U_\BA$, and 
$\BA[2P]^{W\circ}$ is spanned over $\BA$ by $\iota(t_{L(\lambda)})$ for $\lambda\in P^+$.
\end{proof}

\subsection{}
Let $z\in\BC^\times$.
We denote by $Z_{\Har}(U_z)$ the image of $Z(U_\BA)\to Z(U_z)$, and call it the Harish-Chandra center of $U_z$.
We can similarly consider the composite of
\[
Z_{\Har}(U_z)\hookrightarrow U_z\cong U_z^-\otimes U_z^0\otimes U_z^+\xrightarrow{\varepsilon\otimes1\otimes\varepsilon} U_z^0
\cong\BC[P].
\]
We define an action $\circ_z$ of $W$ on $\BC[2P]$ by
\[
w\circ_z e(2\lambda)=z^{(w\lambda-\lambda,2\tilde{\rho})} e(2w\lambda)\qquad
(w\in W,\;\lambda\in P).
\]

\begin{proposition}
\label{prop:HCzeta}
The above linear map $Z_{\Har}(U_z)\to\BC[P]$ induces an isomorphism 
\[
\iota_z:
Z_{\Har}(U_z)\to\BC[2P]^{W\circ_z}
\]
of $\BC$-algebras.
\end{proposition}
\begin{proof}
By $Z(U_\BA)=U_\BA\cap Z(U)$ 
the canonical map
$\BC\otimes_\BA Z(U_\BA)\to U_z$ is injective.
Hence $Z_{\Har}(U_z)\cong \BC\otimes_\BA Z(U_\BA)\cong\BC[2P]^{W\circ_z}$.
\end{proof}

For $M\in\Mod(\dot{U}_z)$
we can similarly define $t_M\in {}^eU_z$ by
\[
\omega_z(t_M,x)
=\trace(xk_{2\rho},M)\qquad(x\in\dot{U}_z).
\]
By our construction $\{t_{L_z(\lambda)}\mid\lambda\in P^+\}$ is a basis of $Z_{\Har}(U_z)$.
Indeed for $M\in\Mod(\dot{U}_z)$
we can write 
\[
[M]=\sum_{\lambda\in P^+}m_\lambda[L_z(\lambda)]
\qquad(m_\lambda\in\BZ)
\]
in an appropriate Grothendieck group, and in this case we have
\[
t_M=\sum_{\lambda\in P^+}m_\lambda t_{L_z(\lambda)}\in Z_{\Har}(U_z).
\]

Note that for $z\in\BC^\times$ the two actions $\circ_z$ and $\circ_{-z}$ of $W$ on $\BC[2P]$ are the same.
By Proposition \ref{prop:ZU} we have the following.
\begin{proposition}
For $z\in\BC^\times$ we have 
$U_z^{\CG}
\supset
Z_{\Har}(U_z),
$
and the isomorphism $\Xi_z:U_{-z}^\CG\to U_{z}^\CG$ induces the isomorphism
$Z_{\Har}(U_{-z})
\cong
Z_{\Har}(U_z)$ given by
\[
Z_{\Har}(U_{-z})
\xrightarrow{\iota_{-z}}
\BC[2P]^{W\circ_{-z}}
=
\BC[2P]^{W\circ_{z}}
\xleftarrow{\iota_{z}}
Z_{\Har}(U_z)
\]
\end{proposition}
\subsection{}
We consider the case where $z=1$.
Since the action $\circ_1$ of $W$ on $\BC[2P]$ is nothing but the ordinary one, we have
\[
Z_{\Har}(U_1)\cong 
\BC[2P]^{W}
\cong \BC[P]^W\cong
\BC[H]^W
\cong
\BC[H/W].
\]
Here the second isomorphism is induced by $\BC[2P]\ni e(2\lambda)\leftrightarrow e(\lambda)\in\BC[P]$.
Recall also that we have an isomorphism
\[
U_1\cong \BC[K].
\]
Hence the inclusion $Z_{\Har}(U_1)\to U_1$ induces a morphism $f:K\to H/W$ of algebraic varieties.
Let us give an explicit description of this morphism.
Define a morphism
$
\kappa:K\to G
$
of algebraic varieties by
$\kappa((g_1, g_2)=g_1g_2^{-1}$.
We also define $\upsilon:G\to H/W$ as follows.
Let $g\in G$. 
Let $g_s$ be the semisimple part of $g$ with respect to the Jordan decomposition.
Then $\Ad(G)(g_s)\cap H$ consists of a single $W$-orbit.
We define $\upsilon(g)$ to be this $W$-orbit.

\begin{proposition}[\cite{DP}]
\label{prop:DP-Frob}
The morphism $f:K\to H/W$ is the composite of $\kappa:K\to G$ and $\upsilon:G\to H/W$.
\end{proposition}
\begin{proof}
For the convenience of the readers we give a sketch of the proof using the bilinear forms $\omega_1$ ant $\theta_1$.
First note that
\[
Z_{\Har}(U_1)\subset{}^fU_1\subset{}^eU_1\subset U_1.
\]
Via $\omega_1:{}^eU_1\times\dot{U}_1\to\BC$ we obtain embeddings ${}^fU_1\subset{}^eU_1\subset(\dot{U}_1)^*$.
Identifying $\dot{U}_1$ with the modified enveloping algebra of $\Lie(G)$ we have ${}^fU_1\cong \BC[G]$ (see \cite{C}).
On the other hand we see from $\dot{U}_1\cong \bigoplus_{\lambda\in P}U_1^{L,-}1_\lambda U_1^{L,+}$ that 
${}^eU_1$ is identified with $\BC[N^-\times H\times N^+]$.
Consequently we obtain a sequence 
\[
\BC[H/W]\to\BC[G]\to\BC[N^-\times H\times N^+]\to\BC[K]
\]
of algebra embeddings.
We can easily check that the corresponding morphisms of algebraic varieties are given by 
\begin{align*}
&K\ni(g_+g_0,g_-g_0^{-1})\mapsto (g_-,g_0^{-2},g_+^{-1})\in N^-\times H\times N^+
\quad(g_\pm\in N^\pm, g_0\in H),\\
&N^-\times H\times N^+\ni(x_-,x_0,x_+)\mapsto x_-x_0x_+\in G,\\
&G\ni g\mapsto\upsilon(g)^{-1}\in H/W.
\end{align*}
\end{proof}

\section{Frobenius center}
\subsection{}
Fix a positive integer $\ell$.
If $\ell$ is odd (resp.\ even), then we set $r=\ell$ (resp.\ $r=\ell/2$).
Note that $r$ is the order of $\zeta^2$.
We assume 
\begin{equation}
r>d
\end{equation}
in the following.
We take $\zeta\in\BC$ to be a primitive $\ell$-th root of 1.
Define
$\zeta_\beta\;(\beta\in\Delta)$,
$\zeta_i\;(i\in I)$
as in \eqref{eq:kappa-beta} for $z=\zeta$.
For $\beta\in\Delta$
we denote the orders of 
$\zeta_\beta, \zeta_\beta^2$ by $\ell_\beta, r_\beta$ respectively.
For $i\in I$ we set
$\ell_i=\ell_{\alpha_i}$, $r_i=r_{\alpha_i}$.

\subsection{}
For $\alpha\in\Delta$ set $\alpha'=r_\alpha\alpha\in\Gh_\BQ^*$.
Then ${\Delta}'=\{r_\alpha\alpha\mid\alpha\in\Delta\}$ is a root system with $\{\alpha_i'\mid i\in I\}$ a set of simple roots.
Note that as an abstract root system (disregarding the inner product) we have $\Delta'\cong\Delta$ or $\Delta'\cong\Delta^\vee$.
Set
\begin{align*}
{P}'=
\{\mu\in \Gh_\BQ^*\mid (\mu,\alpha^\vee)\in r_\alpha\BZ\quad(\forall\alpha\in\Delta)\}.\end{align*}
Then $P'$ is the weight lattice for $\Delta'$, and we have
$P'\subset P$.

Set
\begin{equation}
\varepsilon=
\zeta_\alpha^{r_\alpha^2}\qquad
(\alpha\in\Delta,\;\alpha'\in(\Delta')_{\sh}).
\end{equation}
Then we have $\varepsilon=\pm1$.
Furthermore, $\varepsilon=-1$ if and only if we have either 
\begin{itemize}
\item[(a)]
$r$ is odd and $\ell=2r$,
\end{itemize}
or
\begin{itemize}
\item[(b)]
$d=2$, $r$ is even with $r/2$ odd.
\end{itemize}
Set
\[
\varepsilon_{\alpha'}=\varepsilon^{(\alpha',\alpha')/(\beta',\beta')}\qquad
(\alpha'\in\Delta',\quad\beta'\in(\Delta')_{\sh}).
\]
Then we have 
\begin{equation}
\varepsilon_{\alpha'}=\zeta_\alpha^{r_\alpha^2}\qquad(\alpha\in \Delta).
\end{equation}

An explicit description of $(\Delta', \varepsilon)$ in each case is given in Table \ref{tab}.

\begin{table}[h]
\begin{center}
\caption{}
\label{tab}
\begin{tabular}{|l|c|c|c|r|}
\hline
\quad\;\; type of $\Delta$&$\ell$&$r$&$\Delta'$&$\varepsilon\;\,$\\
\hline
\multirow{3}{3.3cm}{$A_n, D_n, E_6, E_7, E_8$}&$\ell\in2\BZ+1$&$\ell$&$r\Delta$&$1$\\
\cline{2-5}
&$\ell\in4\BZ$&${\ell}/2$&$r\Delta$&$1$\\
\cline{2-5}
&$\ell\in4\BZ+2$&${\ell}/2$&$r\Delta$&$-1$\\
\hline
\multirow{4}{1pt}{$\qquad B_n, C_n, F_4$}&$\ell\in2\BZ+1$&$\ell$&$r\Delta$&$1$\\
\cline{2-5}
&$\ell\in4\BZ+2$&${\ell}/2$&$r\Delta$&$-1$\\
\cline{2-5}
&$\ell\in8\BZ$&${\ell}/2$&$\frac{r}2(2\Delta_{\sh}\sqcup\Delta_{\lo})$&$1$\\[1pt]
\cline{2-5}
&$\ell\in8\BZ+4$&${\ell}/2$&$\frac{r}2(2\Delta_{\sh}\sqcup\Delta_{\lo})$&$-1$\\[1pt]
\hline
\multirow{6}{1pt}{$\qquad\quad\;\; G_2$}&$\ell\in6\BZ\pm1$&$\ell$&$r\Delta$&$1$\\
\cline{2-5}
&$\ell\in6\BZ+3$&${\ell}$&$\frac{r}3(3\Delta_{\sh}\sqcup\Delta_{\lo})$&$1$\\[1pt]
\cline{2-5}
&$\ell\in12\BZ$&${\ell}/2$&$\frac{r}3(3\Delta_{\sh}\sqcup\Delta_{\lo})$&$1$\\[1pt]
\cline{2-5}
&$\ell\in12\BZ\pm4$&${\ell}/2$&$r\Delta$&$1$\\
\cline{2-5}
&$\ell\in12\BZ+6$&${\ell}/2$&$\frac{r}3(3\Delta_{\sh}\sqcup\Delta_{\lo})$&$-1$\\[1pt]
\cline{2-5}
&$\ell\in12\BZ\pm2$&${\ell}/2$&$r\Delta$&$-1$\\
\hline
\end{tabular}
\end{center}
\end{table}

\subsection{}
Similarly to the Frobenius homomorphism
\begin{equation}
\label{eq:Fr}
\Fr:\dot{U}_\zeta(\Delta)\to\dot{U}_{\varepsilon}(\Delta')
\end{equation}
given in \cite[Theorem 35.1.9]{Lbook} we can define an algebra homomorphism
\begin{equation}
\xi:\dot{V}_\zeta(\Delta)\to\dot{V}_{\varepsilon}(\Delta')
\end{equation}
such that 
\begin{itemize}
\item 
for $\lambda\notin P'$ we have
$\xi(x_i^{(n)}1_\lambda)=\xi(y_i^{(n)}1_\lambda)=0$\quad($i\in I, n\in\BZ_{\geqq0}$),
\item 
for $\lambda\in P'$ we have
\begin{align*}
\xi(x_i^{(n)}1_\lambda)&=
\begin{cases}
x_i^{(n/r_i)}1_\lambda\qquad&(r_i| n)\\
0\qquad&(\text{otherwise}),\\
\end{cases}
\\
\xi(y_i^{(n)}1_\lambda)&=
\begin{cases}
y_i^{(n/r_i)}1_\lambda\qquad&(r_i| n)\\
0\qquad&(\text{otherwise}).
\end{cases}
\end{align*}
\end{itemize}
The fact that $\xi$ is well-defined follows easily from the corresponding fact for 
$
\Fr
$.
Moreover, for
$\lambda\in P'$ and $\beta\in\Delta^+$
\begin{align*}
\xi(x_\beta^{(n)}1_\lambda)&=
\begin{cases}
x_{\beta'}^{(n/r_\beta)}1_\lambda\qquad&(r_\beta| n)\\
0\qquad&(\text{otherwise}),\\
\end{cases}
\\
\xi(y_\beta^{(n)}1_\lambda)&=
\begin{cases}
y_{\beta'}^{(n/r_\beta)}1_\lambda\qquad&(r_\beta| n)\\
0\qquad&(\text{otherwise})
\end{cases}
\end{align*}
by
\cite[41.1.9]{Lbook}.
\begin{proposition}
There exists uniquely an injective homomorphism 
\[
{}^t\xi:U_\varepsilon(\Delta')\to U_\zeta(\Delta)
\]
of coalgebras satisfying
\begin{equation}
\label{eq:txi1}
\sigma_\zeta({}^t\xi(u),v)=
\sigma_\varepsilon(u,\xi(v))
\qquad
(u\in U_\varepsilon(\Delta'), v\in\dot{V}_\zeta(\Delta)).
\end{equation}
Moreover, we have
\begin{align}
\label{eq:txi2}
&{}^t\xi(a_{\beta'_N}^{n_N}\cdots a_{\beta'_1}^{n_1}
k_\mu
S(b_{\beta'_N}^{n'_N}\cdots b_{\beta'_1}^{n'_1}))\\
\nonumber
=&
c_{\beta_1}^{n_1+n'_1}\cdots
c_{\beta_N}^{n_N+n'_N}
a_{\beta_N}^{r_{\beta_N}n_N}\cdots a_{\beta_1}^{r_{\beta_1}n_1}
k_{\mu}
S(b_{\beta_N}^{r_{\beta_N}n'_N}\cdots b_{\beta_1}^{r_{\beta_1}n'_1})\\
\nonumber
&\hspace{3cm}
(\mu\in P', n_1,\dots, n_N, n'_1,\dots, n'_N\in\BZ_{\geqq0}),
\end{align}
where
\begin{align*}
c_\beta
&=
(-1)^{r_\beta+1}
\zeta_{\beta}^{-r_\beta(r_\beta-1)/2}
\qquad(\beta\in\Delta^+).
\end{align*}
\end{proposition}
\begin{proof}
It is easily seen from \eqref{eq:Drinfeld-value} that there exists uniquely a linear map
$
{}^t\xi:U_\varepsilon(\Delta')\to U_\zeta(\Delta)
$
satisfying \eqref{eq:txi1}, and it is given by \eqref{eq:txi2}.
Then we conclude from Proposition \ref{prop:sigma-inv}  that 
${}^t\xi$ is a homomorphism of coalgebras.
\end{proof}
Similarly we have the following.
\begin{proposition}
\label{prop:omagaA}
We have 
${}^t\xi({}^eU_\varepsilon(\Delta'))\subset{}^eU_\zeta(\Delta)$, and
\[
\omega_\zeta({}^t\xi(u),x)=
\omega_\varepsilon(u,\Fr(x))
\qquad
(u\in {}^eU_\varepsilon(\Delta'), x\in\dot{U}_\zeta(\Delta)).
\]
\end{proposition}

\subsection{}
For
$\beta\in\Delta$ we set
$\eta_\beta=\zeta_\beta^{r_\beta}$.
We have $\eta_\beta=\pm1$, and $\eta_\beta=-1$ if and only $\ell_\beta$ is even.
\begin{proposition}[De Concini-Kac \cite{DK}]
\label{prop:rel-zeta}
For 
$\alpha, \beta\in\Delta^+$, $\lambda\in P$, $\mu\in P'$
we have
\begin{align*}
&a_{\alpha}^{r_\alpha}a_{\beta}=
\eta_\alpha^{(\alpha^\vee,\beta)}
a_{\beta}a_{\alpha}^{r_\alpha},\qquad
(Sb_{\alpha}^{r_\alpha})(Sb_{\beta})=
\eta_\alpha^{(\alpha^\vee,\beta)}
(Sb_{\beta})(Sb_{\alpha}^{r_\alpha}),\\
&a_{\alpha}^{r_\alpha}(Sb_{\beta})=
\eta_\alpha^{(\alpha^\vee,\beta)}
(Sb_{\beta})a_{\alpha}^{r_\alpha},\qquad
(Sb_{\alpha}^{r_\alpha})a_{\beta}=
\eta_\alpha^{(\alpha^\vee,\beta)}
a_{\beta}(Sb_{\alpha}^{r_\alpha}),\\
&k_\lambda a_\alpha^{r_\alpha}
=\eta_\alpha^{(\lambda,\alpha^\vee)}
a_\alpha^{r_\alpha}k_\lambda,\qquad
k_\lambda (Sb_\alpha^{r_\alpha})
=\eta_\alpha^{(\lambda,\alpha^\vee)}
(Sb_\alpha^{r_\alpha})k_\lambda,\\
&k_{\mu}a_\alpha
=\eta_\alpha^{(\mu,\alpha^\vee)/r_\alpha}
a_\alpha k_{\mu}
,\qquad
k_{\mu} (Sb_\alpha)
=\eta_\alpha^{(\mu,\alpha^\vee)/r_\alpha}
(Sb_\alpha)k_{\mu}
\end{align*}
in $U_\zeta(\Delta)$.
\end{proposition}
\begin{proposition}
\label{prop:rel-epsilon}
For
$\alpha', \beta'\in(\Delta')^+$, $\mu\in P'$ 
we have
\begin{align*}
&a_{\alpha'}a_{\beta'}=
\varepsilon_{\alpha'}^{((\alpha')^\vee,\beta')}
a_{\beta'}a_{\alpha'},\qquad
(Sb_{\alpha'})(Sb_{\beta'})=
\varepsilon_{\alpha'}^{((\alpha')^\vee,\beta')}
(Sb_{\beta'})(Sb_{\alpha'}),\\
&a_{\alpha'}(Sb_{\beta'})=
\varepsilon_{\alpha'}^{((\alpha')^\vee,\beta')}
(Sb_{\beta'})a_{\alpha'},\qquad
(Sb_{\alpha'})a_{\beta'}=
\varepsilon_{\alpha'}^{((\alpha')^\vee,\beta')}
a_{\beta'}(Sb_{\alpha'}),\\
&k_{\mu}a_{\alpha'}
=\varepsilon_{\alpha'}^{(\mu,(\alpha')^\vee)}
a_{\alpha'} k_{\mu}
,\qquad
k_{\mu} (Sb_{\alpha'})
=\varepsilon_{\alpha'}^{(\mu,(\alpha')^\vee)}
(Sb_{\alpha'})k_{\mu}
\end{align*}
in $U_\varepsilon(\Delta')$.
\end{proposition}
\begin{proof}
Let
\[
(\,|\,)':\Gh_\BQ\times\Gh_\BQ\to\BQ
\]
be the $W$-invariant non-degenerate symmetric bilinear form such that $(\alpha'|\alpha')'=2$ for $\alpha'\in(\Delta')_{\sh}$.
Then we have $\varepsilon_{\alpha'}^{((\alpha')^\vee,\beta')}=\varepsilon^{(\alpha'|\beta')'}$ for $\alpha', \beta'\in(\Delta')^+$.

In order to show the first formula 
$a_{\alpha'}a_{\beta'}=
\varepsilon^{(\alpha'|\beta')'}
a_{\beta'}a_{\alpha'}$ for
$\alpha', \beta'\in(\Delta')^+$, 
it is sufficient to show
\[
\tau^{\emptyset,L}_\epsilon(a_{\alpha'}a_{\beta'},y)=
\varepsilon^{(\alpha'|\beta')'}
\tau^{\emptyset,L}_\epsilon(a_{\beta'}a_{\alpha'},y)
\]
for any $y\in U_\varepsilon^{L,-}=U_\varepsilon^{L,-}(\Delta')$,
where $\tau^{\emptyset,L}_\epsilon$ is defined for $\Delta'$.
Write
\[
\Delta(y)=\sum_{\gamma,\delta\in(Q')^+}u^y_{\gamma,\delta}(k_\delta^{-1}\otimes1)
\qquad(u^y_{\gamma,\delta}\in U^{L,-}_{\varepsilon,-\gamma}\otimes U^{L,-}_{\varepsilon,-\delta}),
\]
where for $\gamma=\sum_{i\in I}m_{i}\alpha_i'\in(Q')^+$
we set
\[
U^{L,-}_{\varepsilon,-\gamma}
=
\sum_{\sum_{k_j=i}n_{j}=m_i}\BC f_{k_1}^{(n_{1})}\cdots f_{k_s}^{(n_{s})}
\subset
U^{L,-}_{\varepsilon}.
\]
Then we have
\begin{align*}
&\tau^{\emptyset,L}_\epsilon(a_{\alpha'}a_{\beta'},y)
=
(\tau^{\emptyset,L}_\epsilon\otimes\tau^{\emptyset,L}_\epsilon)(a_{\beta'}\otimes a_{\alpha'},\Delta(y))\\
=&(\tau^{\emptyset,L}_\epsilon\otimes\tau^{\emptyset,L}_\epsilon)(a_{\alpha'}\otimes a_{\beta'},\CP(u_{\beta',\alpha'}^y)),
\end{align*}
where $\CP(y_1\otimes y_2)=y_2\otimes y_1$.
Similarly, we have
\[
\tau^{\emptyset,L}_\epsilon(a_{\beta'}a_{\alpha'},y)
=(\tau^{\emptyset,L}_\epsilon\otimes\tau^{\emptyset,L}_\epsilon)(a_{\alpha'}\otimes a_{\beta'},u_{\alpha',\beta'}^y).
\]
Hence it is sufficient to show
\begin{equation}
\label{eq:com}
\CP(u_{\gamma,\delta}^y)
=
\varepsilon^{(\gamma|\delta)'}u_{\delta,\gamma}^y
\qquad(y\in U^{L,-}_\varepsilon,\;\gamma, \delta\in (Q')^+).
\end{equation}
We can easily check that if \eqref{eq:com} holds for $y=y_1, y_2$, then it also holds for $y=y_1y_2$.
Hence the assertion follows from \eqref{eq:com} for $y=f_i^{(n)}$, which is easily checked.

The second formula is equivalent to 
$b_{\alpha'}b_{\beta'}=
\varepsilon^{(\alpha'|\beta')'}
b_{\beta'}b_{\alpha'}$
for $\alpha', \beta'\in(\Delta')^+$, and is proved similarly to the first formula.

Let us show the third and the fourth formula.
They are equivalent to
$a_{\alpha'}b_{\beta'}=
b_{\beta'}a_{\alpha'}$ for
$\alpha', \beta'\in(\Delta')^+$.
Take $1\leqq j,k\leqq N$ such that
$\alpha'=\beta_j', \beta'=\beta_k'$.
If $j=k$, then the assertion is a consequence of 
$a_{\alpha_i'}b_{\alpha_i'}=
b_{\alpha_i'}a_{\alpha_i'}$ in $U_\varepsilon(\Delta')$ for
$i\in I$.
Assume $j>k$.
Setting
\[
w=s_{i_1}\cdots s_{i_{k-1}},\qquad
y=s_{i_k}\cdots s_{i_{j-1}},\qquad
i_j=m,\qquad
i_k=n
\]
we have
\[
a_{\alpha'}=T_wT_y(a_{\alpha'_m}),\qquad
b_{\beta'}=T_w(b_{\alpha'_n}),
\]
and hence it is sufficient to show
\[
b_{\alpha'_n}T_y(a_{\alpha'_m})=T_y(a_{\alpha'_m})b_{\alpha'_n}.
\]
By $s_ny<y$ this is equivalent to
\[
T_n^{-1}(b_{\alpha'_n})T_{s_ny}(a_{\alpha'_m})=T_{s_ny}(a_{\alpha'_m})T_n^{-1}(b_{\alpha'_n}).
\]
By
$T_n^{-1}(b_{\alpha'_n})=-a_{\alpha'_n}k_n$ this is again equivalent to
\[
a_{\alpha'_n}T_{s_ny}(a_{\alpha'_m})=
\varepsilon^{(\alpha_n'|s_ny(\alpha'_m))'}
T_{s_ny}(a_{\alpha'_m})a_{\alpha'_n}.
\]
By $s_ny<s_nys_m$ we have $s_ny(\alpha'_m)\in(\Delta')^+$ and
$T_{s_ny}(a_{\alpha'_m})$ is a linear combination of the elements of the form
$
a_{\beta'_N}^{m_N}\cdots a_{\beta'_1}^{m_1}
$
with $\sum_jm_j\beta'_j=s_ny(\alpha'_m)$.
Hence the assertion follows from the first formula.
The case $j<k$ can be handled in a similar way.

The remaining formulas are obvious.
\end{proof}

We see easily from Proposition \ref{prop:rel-zeta}, Proposition \ref{prop:rel-epsilon} the following.
\begin{proposition}
${}^t\xi$ is a homomorphism of Hopf algebras.
\end{proposition}

\subsection{}
We define the Frobenius center  $Z_\Fr(U_\zeta)$ of  $U_\zeta$ by $Z_\Fr(U_\zeta)=\Image({}^t\xi)\cap Z(U_\zeta)$.
Note
\[
\Image({}^t\xi)\cong
\left(\bigotimes_{\alpha\in\Delta^+}\BC[a_\alpha^{r_\alpha},Sb_\alpha^{r_\alpha}]\right)
\otimes
\BC[P'].
\]
Namely, the image of ${}^t\xi$ consists of the linear combinations of the monomials of the form
\begin{equation}
\label{eq:monomial}
z=
a_{\beta_1}^{r_{\beta_1}m_{\beta_1}}\cdots
a_{\beta_N}^{r_{\beta_N}m_{\beta_N}}
k_{\mu}
(Sb_{\beta_1}^{r_{\beta_1}m'_{\beta_1}})\cdots
(Sb_{\beta_N}^{r_{\beta_1}m'_{\beta_N}})
\qquad(\mu\in P').
\end{equation}

If $\ell$ is odd, then we have $\eta_\alpha=1$ for any $\alpha\in\Delta^+$, and hence
$Z_\Fr(U_\zeta)=\Image({}^t\xi)$ by Proposition \ref{prop:rel-zeta}.

Assume $\ell$ is even.
By Proposition \ref{prop:rel-zeta} we see easily that 
$Z_\Fr(U_\zeta)$ consists of the linear combinations of the monomials of the form \eqref{eq:monomial} satisfying
\begin{align}
\label{eq:monomial1}
&
\sum_{\alpha\in\Delta^+_1}(m_\alpha+m_\alpha')\alpha^\vee\in 2Q^\vee,
\\
\label{eq:monomial2}
&
(\mu,\gamma^\vee)/r_\gamma
\in 2\BZ\qquad(\forall\gamma\in\Delta^+_1),
\end{align}
where
\[
\Delta^+_1=\{\alpha\in\Delta^+\mid
\eta_\alpha=-1\}
=
\begin{cases}
\Delta_\sh\cap\Delta^+\;&(r\not\in2\BZ, \ell=2r, d=2)\\
\Delta^+\;&(\text{otherwise}).
\end{cases}
\]
Note that \eqref{eq:monomial2} is equivalent to $\mu\in P''$, where
\begin{equation}
\label{eq:monomial2A}
P''=
\begin{cases}
2P'_0\;&(r\not\in2\BZ, \ell=2r, d=2)\\
2P'\;&(\text{otherwise}).
\end{cases}
\end{equation}
Here 
\[
P'_0=\{\lambda\in\Gh_\BQ^*\mid
d_\alpha(\lambda,\alpha^\vee)\in r\BZ\quad(\alpha\in\Delta)\}.
\]

Define subgroups  $\Gamma_1$ and $\Gamma_2$ of 
$\CG(\Delta')$ and $\CH(\Delta')$ respectively 
by
\begin{align*}
\Gamma_1&=
\begin{cases}
\{1\}&(\ell\not\in2\BZ)\\
\CG(\Delta')
\quad&(\ell\in2\BZ),
\end{cases}
\\
\Gamma_2&=
\begin{cases}
\{1\}&(\ell\not\in2\BZ)\\
(Q')^\vee_\sh/2(Q')^\vee_\sh
\quad&(\ell\in2\BZ, r\not\in2\BZ, d=2)\\
\CH(\Delta')
\qquad&(\text{otherwise}),
\end{cases}
\end{align*}
where
\[
(Q')^\vee_\sh
=
\sum_{\alpha'\in(\Delta')_\sh}\BZ(\alpha')^\vee.
\]
Set
\[
\Gamma=\Gamma_1\times \Gamma_2.
\]

By the above argument we have the following.
\begin{proposition}
\label{prop:Frob}
Under the identification $\Image({}^t\xi)\cong U_\varepsilon(\Delta')$ we have
\begin{align*}
Z_{\Fr}(U_\zeta(\Delta))\cong &
(U^+_\varepsilon(\Delta')\otimes SU^-_\varepsilon(\Delta'))^{\Gamma_1}
\otimes\BC[P'']
\\
=&
(U^+_\varepsilon(\Delta')\otimes SU^-_\varepsilon(\Delta'))^{\Gamma_1}
\otimes U^0_\varepsilon(\Delta')^{\Gamma_2}\\
=&
U_\varepsilon(\Delta')^{{\Gamma}}.
\end{align*}
\end{proposition}
\begin{proposition}
\label{prop:Frob2}
We have an isomorphism 
\begin{equation}
\label{eq:Z0}
Z_{\Fr}(U_\zeta(\Delta))\cong\BC[K(\Delta')]^{{\Gamma}}\;(=\BC[K(\Delta')/{\Gamma}])
\end{equation}
of algebras.
\end{proposition}
\begin{proof}
Assume $\varepsilon=1$.
Then the action of the group ${\Gamma}$ on the algebra $U_1(\Delta')$ induces the action of ${\Gamma}$ on the algebraic variety $K(\Delta')$
via  the algebra isomorphism \eqref{eq:U1K} for $\Delta'$, and hence we have
$U_1(\Delta')^\Gamma\cong\BC[K(\Delta')]^\Gamma\cong\BC[K(\Delta')/\Gamma]$.
Assume $\varepsilon=-1$.
In this case we have
$U_{-1}(\Delta')^\Gamma\cong
U_1(\Delta')^\Gamma$ by Proposition \ref{prop:Xi}.
Hence we have also 
$U_{-1}(\Delta')^\Gamma\cong\BC[K(\Delta')]^\Gamma\cong\BC[K(\Delta')/\Gamma]$.
\end{proof}

By Proposition \ref{prop:Frob} and \cite{HE}
we obtain the following.
\begin{corollary}
\label{cor:CM}
$Z_\Fr(U_\zeta)$ is Cohen-Macaulay.
\end{corollary}

\section{Main result}
Since the action $\circ_{\varepsilon}$ of $W$ on $\BC[2P']$ is the ordinary one, we have
\begin{equation}
\label{eq:Z1}
Z_\Har(U_\varepsilon(\Delta'))\cong
\BC[2P']^W\cong\BC[P']^W\cong
\BC[H(\Delta')/W],
\end{equation}
where the second isomorphism is induced by $\BC[P']\cong\BC[2P']\;(e(\lambda)\leftrightarrow e(2\lambda))$.
Similarly, we have
\begin{equation}
\label{eq:Z2}
Z_\Har(U_\zeta(\Delta))\cong\BC[H(\Delta)/W].
\end{equation}
Note that the action of $W$ on $H(\Delta')$ in \eqref{eq:Z1} is the ordinary one, while that on $H(\Delta)$ in \eqref{eq:Z2} is the twisted one given by 
\[
w:h\mapsto w(h_1h)h_1^{-1}
\qquad(w\in W, h\in H(\Delta)),
\]
where $h_1\in H(\Delta)$ is given by $\lambda(h_1)=\zeta^{2(\lambda,\tilde{\rho})}\;(\lambda\in P=\Hom(H(\Delta),\BC^\times))$.
\begin{proposition}
We have
\[
Z_\Fr(U_\zeta(\Delta))\cap Z_\Har(U_\zeta(\Delta))
={}^t\xi(Z_\Har(U_\varepsilon(\Delta'))),
\]
and hence 
\begin{equation}
\label{eq:Z3}
Z_\Fr(U_\zeta(\Delta))\cap Z_\Har(U_\zeta(\Delta))\cong\BC[H(\Delta')/W].
\end{equation}
\end{proposition}
\begin{proof}
Note that $Z_\Har(U_\varepsilon(\Delta'))$ is spanned by $\{t_{L_\varepsilon(\lambda)}\}_{\lambda\in (P')^+}$.
By Proposition \ref{prop:omagaA} we have ${}^t\xi(t_{L_\varepsilon(\lambda)})=t_{\Fr^*L_\varepsilon(\lambda)}$, where $\Fr^*L_\varepsilon(\lambda)$ is the $\dot{U}_\zeta(\Delta)$-module induced via $\Fr:\dot{U}_\zeta(\Delta)\to\dot{U}_\varepsilon(\Delta')$.
Hence we have 
\[
Z_\Fr(U_\zeta(\Delta))\cap Z_\Har(U_\zeta(\Delta))
\supset{}^t\xi(Z_\Har(U_\varepsilon(\Delta'))),
\]
and
$\iota_\zeta({}^t\xi(Z_\Har(U_\varepsilon(\Delta'))))=\BC[2P']^W$.
On the other hand by Proposition \ref{prop:Frob} we have 
$\iota_\zeta(Z_\Fr(U_\zeta(\Delta))\cap Z_\Har(U_\zeta(\Delta)))\subset
\BC[2P]^{W\circ_\zeta}\cap\BC[P'']=
\BC[2P']^W$.
\end{proof}
By the definition of the Harish-Chandra isomorphism we have the following.
\begin{proposition}
The morphism $H(\Delta)/W\to H(\Delta')/W$, which is associated to the inclusion
$Z_\Fr(U_\zeta(\Delta))\cap Z_\Har(U_\zeta(\Delta))\subset
Z_\Har(U_\zeta(\Delta))$ together with the isomorphisms \eqref{eq:Z2} and \eqref{eq:Z3}, 
is the natural one induced from the canonical morphism $H(\Delta)\to H(\Delta')$ associated to the embedding 
$P'\subset P$.
\end{proposition}

Note that we have the following commutative diagram
\[
\begin{CD}
Z_\Fr(U_\zeta(\Delta))\cap Z_\Har(U_\zeta(\Delta))
@>>>
Z_{\Fr}(U_\zeta(\Delta))
\\
@AAA@AAA
\\
Z_{\Har}(U_\varepsilon(\Delta'))
@>>>
U_\varepsilon(\Delta')^\Gamma
\\
@AAA@AAA
\\
Z_{\Har}(U_1(\Delta'))
@>>>
U_1(\Delta')^\Gamma
@>>>
U_1(\Delta')
\\
@AAA@AAA@AAA
\\
\BC[H(\Delta')/W]
@>>>
\BC[K(\Delta')/\Gamma]
@>>>
\BC[K(\Delta')]
\end{CD}
\]
where horizontal arrows are inclusions, and vertical arrows are isomorphisms.
Note also that the inclusion 
$\BC[H(\Delta')/W]
\to
\BC[K(\Delta')]$
is induced by $\upsilon\circ\kappa$, where 
$
\kappa:K(\Delta')\to G(\Delta')
$ and 
$\upsilon:G(\Delta')\to H(\Delta')/W$
are morphisms of algebraic varieties we have already defined.
Hence we have the following.
\begin{proposition}
The morphism $K(\Delta')/{\Gamma}\to H(\Delta')/W$, which is associated to the inclusion
$Z_\Fr(U_\zeta(\Delta))\cap Z_\Har(U_\zeta(\Delta))\subset
Z_\Fr(U_\zeta(\Delta))$ together with the isomorphisms \eqref{eq:Z0} and \eqref{eq:Z3}, 
is induced by 
$\upsilon\circ\kappa:K(\Delta')\to H(\Delta')/W$.
\end{proposition}
The main result of this paper is the following.
\begin{theorem}
\label{thm:main}
The natural homomorphism
\[
Z_{\Fr}(U_\zeta)\otimes_{Z_{\Fr}(U_\zeta)\cap Z_{\Har}(U_\zeta)}Z_{\Har}(U_\zeta)
\to
Z(U_\zeta)
\]
is an isomorphism.
In particular, we have
\[
Z(U_\zeta)
\cong
\BC[(K(\Delta')/{\Gamma})\times_{H(\Delta')/W}(H(\Delta)/W)].
\]
\end{theorem}

The rest of the paper is devoted to the proof of Theorem \ref{thm:main}.
The arguments below mostly follow that in De Concini-Kac-Procesi \cite{DKP} (see also De Concini-Procesi \cite{DP}).
We set for simplicity
\begin{align*}
Z&=Z(U_\zeta),\\
Z_{\Fr}&=Z_{\Fr}(U_\zeta)\cong\BC[K(\Delta')/{\Gamma}],\\
Z_{\Har}&=Z_{\Har}(U_\zeta)
\cong\BC[H(\Delta)/W],
\end{align*}
so that
\[
Z_{\Fr}\cap Z_{\Har}\cong\BC[H(\Delta')/W].
\]
We are going to show that the canonical homomorphism
\[
j:Z_{\Fr}\otimes_{Z_{\Fr}\cap Z_{\Har}}Z_{\Har}
\to
Z
\]
is an isomorphism.
\begin{proposition}
\label{prop:normal}
$Z_{\Fr}\otimes_{Z_{\Fr}\cap Z_{\Har}}Z_{\Har}$ is a normal domain.
\end{proposition}
\begin{proof}
By Serre's criterion it is sufficient to show that the scheme
$(K(\Delta')/{\Gamma})\times_{H(\Delta')/W}(H(\Delta)/W)$ is smooth in codimension one and Cohen-Macaulay.

We first show that $(K(\Delta')/{\Gamma})\times_{H(\Delta')/W}(H(\Delta)/W)$ is smooth in codimension one.
Since $H(\Delta)/W$ is smooth and $H(\Delta)/W\to H(\Delta')/W$ is a finite morphism, it is sufficient to show that there exists a subvariety $X$ of $K(\Delta')/{\Gamma}$ with codimension greater than one such that 
$(K(\Delta')/{\Gamma})\setminus X\to H(\Delta')/W$ is smooth.
Consider first $K(\Delta')\to H(\Delta')/W$.
Then there exists a subvariety $X_1$ of $K(\Delta')$ with codimension greater than one such that 
$K(\Delta')\setminus X_1\to H(\Delta')/W$ is smooth since a similar result is known to hold for $G(\Delta')\to H(\Delta')/W$ and $K(\Delta')\to G(\Delta')$ is smooth.
Hence it is sufficient to show that there exists a subvariety $X_2$ of $K(\Delta')$ with codimension greater than one such that $K(\Delta')\setminus X_2\to (K(\Delta')/{\Gamma}$ is smooth since $K(\Delta')\to (K(\Delta')/{\Gamma}$ is a finite morphism.
We may assume ${\Gamma}\ne\{1\}$.
In this case we have
\[
K(\Delta')=Y\times \Spec\;\BC[P'],\quad
K(\Delta')/{\Gamma}
=Y/{P'}\times \Spec\;\BC[P''],
\]
where
$
Y=\prod_{\alpha\in\Delta^+}\BC^2
$
and the action of $P'$ on $Y$ is given by
\[
\lambda:(x_\alpha)_{\alpha\in\Delta^+}\mapsto
((-1)^{d_{\alpha'}(\lambda,(\alpha')^\vee)}x_\alpha)_{\alpha\in\Delta^+}
\qquad
(\lambda\in P',\;x_\alpha\in\BC^2).
\]
Since $\Spec\;\BC[P']\to\Spec\;\BC[P'']$ is smooth, it is sufficient to show that there exists a subvariety $Z$ of $Y$ with codimension greater than one such that $Y\setminus Z\to Y/{P'}$ is smooth.
Note that the obvious action of $\prod_{\alpha\in\Delta^+}GL_2(\BC)$ on $Y$ commutes with the action of $P'$.
Hence $Y\to Y/{P'}$ is smooth on the open orbit
$
Y_0=\prod_{\alpha\in\Delta^+}(\BC^2\setminus\{0\})
$.
Our claim is a consequence of 
$\dim(Y\setminus Y_0)\leqq \dim Y-2$.

Let us show that $Z_{\Fr}(U_\zeta)\otimes_{\BC[2P']^W}{\BC[2P]^{W\circ_\zeta}}$ is Cohen-Macaulay.
By \cite{St} $\BC[2P']^W$ and $\BC[2P]^{W\circ_\zeta}$ are both isomorphic to the polynomial ring in $|I|$-variables.
Hence we have 
\[
Z_{\Fr}(U_\zeta)\otimes_{\BC[2P']^W}{\BC[2P]^{W\circ_\zeta}}
\cong
Z_{\Fr}(U_\zeta)[X_1,\dots, X_{|I|}]/(f_1,\dots, f_{|I|})
\]
for some $f_1,\dots, f_{|I|}\in Z_{\Fr}(U_\zeta)[X_1,\dots, X_{|I|}]$.
Moreover, we have obviously $\dim Z_{\Fr}(U_\zeta)\otimes_{\BC[2P']^W}{\BC[2P]^{W\circ_\zeta}}=\dim Z_{\Fr}(U_\zeta)$.
Hence our claim is a consequence of Corollary \ref{cor:CM} and well-known results on Cohen-Macaulay rings.
\end{proof}
\begin{lemma}
\label{lem:rank}
$Z_{\Fr}\otimes_{Z_{\Fr}\cap Z_{\Har}}Z_{\Har}$ is a free $Z_{\Fr}$-module of rank $P/P'$.
\end{lemma}
\begin{proof}
It is sufficient to show that 
$Z_{\Har}$ 
is a free $Z_{\Fr}\cap Z_{\Har}$-module of rank $P/P'$.
Namely, we have only to show that 
$\BC[2P]^{W\circ_\zeta}$ 
is a free $\BC[2P']^{W}$-module of rank $P/P'$.
We may replace $\BC[2P]^{W\circ_\zeta}$ with $\BC[2P]^{W}$ by applying an automorphism of $\BC[P]$ which sends $\BC[2P]^{W\circ_\zeta}$ and $\BC[2P']^{W}$ to
$\BC[2P]^{W}$ and $\BC[2P']^{W}$  respectively.
By Steinberg \cite{St} 
$\BC[2P]$ (resp.\ $\BC[2P']$) is a free $\BC[2P]^{W}$-module
(resp.\ $\BC[2P']^{W}$-module) of rank $|W|$.
Since $\BC[2P]$ is a free $\BC[2P']$-module of rank $|P/P'|$, 
$\BC[2P]$ is a free $\BC[2P']^W$-module of rank $|W|\times|P/P'|$.
Note that 
$\BC[2P]^W$ is a direct summand of the free $\BC[2P]^W$-module $\BC[2P]$ by \cite{St}.
Hence $\BC[2P]^W$ is also a direct summand of the free $\BC[2P']^W$-module $\BC[2P]$.
It follows that $\BC[2P]^W$ is a projective $\BC[2P']^W$-module of rank $|P/P'|$.
Since $\BC[2P']^W$ is isomorphic to a polynomial ring by \cite{St}, we conclude 
from the Serre conjecture that 
$\BC[2P]^{W}$ 
is a free $\BC[2P']^{W}$-module of rank $P/P'$.
\end{proof}
Set
\begin{equation}
\label{eq:m}
m=
\begin{cases}
1\qquad&(\ell\not\in 2\BZ)\\
2^{|\Delta_{\sh}\cap \Pi|}
&(\ell\in2\BZ, r\not\in 2\BZ, d=2)\\
2^{|\Pi|}
&(\text{otherwise}).
\end{cases}
\end{equation}
For a commutative domain $S$ we denote by $Q(S)$ the quotient field.
\begin{lemma}
\label{lem:rank2}
$U_\zeta$ is a finitely generated $Z_\Fr$-module, and we have
\[
\dim_{Q(Z_{\Fr})}Q(Z_{\Fr})\otimes_{Z_\Fr}U_\zeta=\left(m\prod_{\alpha\in\Delta^+}r_\alpha\right)^2\times |P/P'|.
\]
\end{lemma}
\begin{proof}
Denote by $C$ the image of ${}^t\xi:U_\varepsilon(\Delta')\to U_\zeta(\Delta)$.
Then we have
\[
Z_{\Fr}\subset C\subset U_\zeta.
\]
Since $U_\zeta$ is a free $C$-module of rank $\left(\prod_{\alpha\in\Delta^+}r_\alpha\right)^2\times |P/P'|$, it is sufficient to show that $C$ is a finitely generated $Z_\Fr$-module and
\[
\dim_{Q(Z_{\Fr})}Q(Z_{\Fr})\otimes_{Z_\Fr}C=m^2.
\]
If $\ell$ is odd, we have $C=Z_\Fr$, and hence we may assume that $\ell$ is even.
By the explicit description of $Z_\Fr$ given by \eqref{eq:monomial1}, \eqref{eq:monomial2}
we have
\[
C\cong
\BC[P']\otimes\BC[(\BZ_{\geqq0}^2)^{\Delta^+}],\qquad
Z_\Fr\cong
\BC[P'']\otimes\BC[L],
\]
where
\[
L=\{(m_\beta,m'_\beta)_{\beta\in\Delta^+}
\in(\BZ_{\geqq0}^2)^{\Delta^+}
\mid
\sum_{\beta\in\Delta_1^+}(m_\beta+m'_\beta)\beta^\vee\in 2Q^\vee\},
\]
and $\BC[(\BZ_{\geqq0}^2)^{\Delta^+}]$ and $\BC[L]$ are the semigroup algebras of the semigroups $(\BZ_{\geqq0}^2)^{\Delta^+}$ and $L$ respectively.
Note that $\BC[P']$ is a free $\BC[P'']$-module of rank $|P'/P''|$.
Since $\BC[(\BZ_{\geqq0}^2)^{\Delta^+}]$ is a finitely generated $\BC[2(\BZ_{\geqq0}^2)^{\Delta^+}]$-module, it is also a finitely generated $\BC[L]$-module by $2(\BZ_{\geqq0}^2)^{\Delta^+}\subset L$.
Hence $C$ is a finitely generated $Z_\Fr$-module.

Set 
\[
\tilde{L}=\{(m_\beta,m'_\beta)_{\beta\in\Delta^+}
\in(\BZ^2)^{\Delta^+}
\mid
\sum_{\beta\in\Delta_1^+}(m_\beta+m'_\beta)\beta^\vee\in 2Q^\vee\}.
\]
Then we have $(\BZ^2)^{\Delta^+}/\tilde{L}\cong Q_1^\vee/(Q_1^\vee\cap 2Q^\vee)$, where $Q_1^\vee=\sum_{\beta\in\Delta_1^+}\BZ\beta^\vee$.
Hence $\BC[(\BZ^2)^{\Delta^+}]$ is a free $\BC[\tilde{L}]$-module of rank $|Q_1^\vee/(Q_1^\vee\cap 2Q^\vee)|$.
Since $\BC[(\BZ^2)^{\Delta^+}]$ and $\BC[\tilde{L}]$ are localizations of $\BC[(\BZ_{\geqq0}^2)^{\Delta^+}]$ and $\BC[L]$ respectively with respect to the multiplicative set $S=2(\BZ_{\geqq0}^2)^{\Delta^+}$ of $\BC[L]$, we obtain that $S^{-1}C$ is a free $S^{-1}Z_\Fr$-module of rank $|P'/P''|\times |Q_1^\vee/(Q_1^\vee\cap 2Q^\vee)|$.
Therefore, $Q(Z_\Fr)\otimes_{Z_\Fr}C$ is a free $Q(Z_\Fr)$-module of rank $|P'/P''|\times |Q_1^\vee/(Q_1^\vee\cap 2Q^\vee)|$.
It remains to show $m^2=|P'/P''|\times |Q_1^\vee/(Q_1^\vee\cap 2Q^\vee)|$.
In the case $\Delta^+_1=\Delta^+$ we have $P''=2P'$, $Q_1^\vee=Q^\vee$, and hence the assertion is obvious.
In the case $\Delta^+_1=\Delta^+\cap\Delta_\sh$ we have $P'=rP$, $P''=r(P\cap 2P_1)$, where
\[
P_1=\{\mu\in \Gh_\BQ^*\mid(\mu,\alpha^\vee)\in\BZ\;(\alpha\in\Delta_\sh)\},
\]
and hence $P'/P''\cong P/(P\cap 2P_1)$.
On the other hand we have
\[
Q_1^\vee/(Q_1^\vee\cap 2Q^\vee)
\cong
(Q_1^\vee+2Q^\vee)/2Q^\vee
\cong
(\frac12Q_1^\vee+Q^\vee)/Q^\vee.
\]
Since $P$ and $P\cap2P_1$ are lattices in $\Gh_\BQ^*$ dual to $Q^\vee$ and $\frac12Q_1^\vee+Q$ respectively, 
we obtain
$|P'/P''|=|Q_1^\vee/(Q_1^\vee\cap 2Q^\vee)|$.
It remains to check $m=|(Q_1^\vee+2Q^\vee)/2Q^\vee|$.
For that it is sufficient to show 
\[
Q_1^\vee+2Q^\vee
=\sum_{\alpha\in\Delta_\sh\cap\Pi}\BZ\alpha^\vee+2Q^\vee.
\]
In order to prove this we have only to show that the right-hand side is stable under the action of the Weyl group.
Hence it is sufficient to show $s_j(\alpha^\vee_i)\in\sum_{\alpha\in\Delta_\sh\cap\Pi}\BZ\alpha^\vee+2Q^\vee$ for any $i, j\in I$ satisfying $\alpha_i\in\Delta_\sh$.
This is obvious if $\alpha_j\in\Delta_\sh$.
In the case where $\alpha_j\in\Delta_\lo$ we have $s_j(\alpha_i^\vee)=
\alpha_i^\vee-(\alpha_i^\vee,\alpha_j)\alpha_j^\vee$ with $(\alpha_i^\vee,\alpha_j)\in\{0,-2\}$.
We are done.
\end{proof}

In general let $R$ be a $\BC$-algebra.
Assume that $R$ is prime (i.e. $x, y\in R,\;xRy=\{0\}$ implies $x=0$ or $y=0$), and is finitely generated as a 
$Z(R)$-module.
Then
$Q(Z(R))\otimes _{Z(R)}R$ is a finite-dimensional central simple algebra over the field $Q(Z(R))$.
Hence $\overline{Q(Z(R))}\otimes _{Z(R)}R$ is isomorphic to the matrix algebra $M_n(\overline{Q(Z(R))})$ for some $n$, where $\overline{Q(Z(R))}$ denotes the algebraic closure of $Q(Z(R))$.
Then this $n$ is called the degree of $R$.
Namely, the degree $n$ of $R$ is given by 
\[
\dim_{Q(Z(R))}Q(Z(R))\otimes _{Z(R)}R=n^2.
\]

Note that $U_\zeta$ is a finitely generated $Z(U_\zeta)$-module by Lemma \ref{lem:rank2}.
In \cite{DK} De Concini-Kac have shown that $U_\zeta$ has no zero divisors using a certain degeneration $\Gr\, U_\zeta$ of $U_\zeta$.
In particular, it is a prime algebra.
Hence we have the notion of the degree of $U_\zeta$.
In \cite{DKP2} De Concini-Kac-Procesi proved that the degree of $U_\zeta$ is less than or equal to that of $\Gr\, U_\zeta$.
They have also shown that the degree of $\Gr\, U_\zeta$ can be computed from the elementary divisors of a certain matrix with integral coefficients.
The actual computation of the elementary divisors was done in \cite{DKP2} when $\ell$ is odd, and in Beck \cite{Beck} in the remaining cases.
From these results we have the following.
\begin{proposition}
\label{prop:rank3}
We have
\[
\dim_{Q(Z)}
Q(Z)\otimes _{Z}U_\zeta\leqq
\left(m\prod_{\alpha\in\Delta^+}r_\alpha\right)^2.
\]
\end{proposition}

Let us show that $j$ is injective.
By Proposition \ref{prop:normal}
$Z_{\Fr}\otimes_{Z_{\Fr}\cap Z_{\Har}}Z_{\Har}$ is a domain.
Note also that $Z$ is a domain since $U_\zeta$ has no zero divisors.
Hence we have only to show that 
\[
j^*:\Spec \;Z\to\Spec\;Z_{\Fr}\otimes_{Z_{\Fr}\cap Z_{\Har}}Z_{\Har}
\]
has a dense image.
Consider the embedding $j':Z_{\Fr}\to Z$.
Since $j'$ is injective, $(j')^*:\Spec \;Z\to\Spec\;Z_{\Fr}$ has a dense image.
Note that $(j')^*$ is the composite of $j^*$ with the natural morphism
\[
\varphi:
\Spec\;Z_{\Fr}\otimes_{Z_{\Fr}\cap Z_{\Har}}Z_{\Har}
\to
\Spec\;Z_{\Fr}.
\]
Since 
$\Spec\;Z_{\Fr}\otimes_{Z_{\Fr}\cap Z_{\Har}}Z_{\Har}$ is irreducible and $\varphi$ is a finite morphism by Lemma \ref{lem:rank}, we conclude that $j^*$ must have a dense image.
The injectivity of $j$ is verified.

Set for simplicity
\[
Z'=Z_{\Fr}\otimes_{Z_{\Fr}\cap Z_{\Har}}Z_{\Har}.
\]
Then we have
\[
Z_{\Fr}\subset Z'\subset Z\subset U_\zeta.
\]
We need to show $Z'=Z$.

Assume that 
\begin{equation}
\label{eq:q5}
Q(Z)=Q(Z')
\end{equation}
holds.
Since 
$U_\zeta$ is a finitely generated $Z_{\Fr}$-module,
$Z$ is a finitely generated $Z'$-module.
It follows that $Z=Z'$ by Proposition \ref{prop:normal}.
Hence it is sufficient to show \eqref{eq:q5}.

Since
$Z'$ is a free $Z_{\Fr}$-module of rank $|P/P'|$, we have
$
[Q(Z'):Q(Z_{Fr})]\geqq |P/P'|$.
Hence it is sufficient to show 
\begin{equation}
\label{eq:q4}
[Q(Z):Q(Z_{Fr})]\leqq|P/P'|.
\end{equation}

Note that we have
$
Q(Z_{\Fr})\otimes_{Z_{\Fr}}Z\cong Q(Z)
$
since $Z$ is a finitely generated $Z_{\Fr}$-module.
Hence
\[
Q(Z_{Fr})\otimes _{Z_{Fr}}U_\zeta
\cong
Q(Z_{Fr})\otimes _{Z_{Fr}}Z\otimes_Z U_\zeta
\cong
Q(Z)\otimes _{Z}U_\zeta.
\]
Hence we obtain \eqref{eq:q4} by Lemma \ref{lem:rank2}, Proposition \ref{prop:rank3}.
The proof of Theorem \ref{thm:main} is complete.

\begin{corollary}
\label{cor:degree}
The degree of $U_\zeta$ is equal to
$m\prod_{\alpha\in\Delta^+}r_\alpha$, where $m$ is as in \eqref{eq:m}.
\end{corollary}
\begin{remark}
{\rm
Corollary \ref{cor:degree} was proved by De Concini-Kac-Procesi \cite{DKP2} in the case $\ell$ is odd and by Beck \cite{Beck} in the case $\ell$ is divided by $4d$.
}
\end{remark}

\bibliographystyle{unsrt}

\end{document}